\documentclass[11pt]{amsart}
\usepackage{amsmath, amsfonts, amsthm, amssymb}
\usepackage[all,cmtip,line]{xy}\usepackage{mathtools}
\usepackage{hyperref} 
\usepackage{array}
\usepackage{enumerate}
\usepackage{srcltx} 

\usepackage{todonotes}
\newtheorem{thm}{Theorem}[section]
\newtheorem{lem}{Lemma}[section]
\newtheorem{prop}[lem]{Proposition}
\newtheorem{cor}[lem]{Corollary}
\newtheorem{q}[lem]{Question}
\newtheorem{defn}[lem]{Definition}
\newtheorem{rem}[lem]{Remark}
\newtheorem{conj}[lem]{Conjecture/Question}

\numberwithin{equation}{section}
\newtheorem{assum}[lem]{Assumption}

\newcommand{\bR}{ \mathbb{R}} 
\newcommand{\bC}{ \mathbb{C}} 

\newcommand{\diam}{ \mbox{diam}}

\newcommand \eps{\varepsilon}

\newcommand \Mp{\mathcal G}

\newlength{\originalbase}
\title{viscosity solution to complex Hessian quotient equation}

\author{Jingrui Cheng, Yulun Xu}
\date{\today}

\begin{document}

\maketitle
\begin{abstract}
In this paper,  we prove the existence of viscosity solutions to complex Hessian equations on compact Hermitian manifolds,  assuming the existence of a strict subsolution in the viscosity sense.  The results cover the complex Hessian quotient equations.
This generalized our previous results in \cite{CX} where the equation needs to satisfy a determinant domination condition.  
\end{abstract}

\section{Introduction}
This paper is a continuation of our previous one \cite{CX} and we wish to generalize our existence results to more general complex Hessian equations.  The Hessian equation we consider takes the following form:
\begin{equation}\label{1.1N}
f\big(\lambda[\chi+dd^c\varphi]\big)=e^{G(x)},\,\,\,\, \lambda[\chi+dd^c\varphi]\in \Gamma.
\end{equation}
In the above,  $G(x)$ is a given continuous function on $M$,  $\chi$ is a given smooth real $(1,1)$ form on $M$,  and $\lambda[\chi+dd^c\varphi]=(\lambda_1,\cdots,\lambda_n)$ denotes the $n$-tuple of eigenvalues of $\chi+dd^c\varphi$ with respect to $\omega_0$,  where $\omega_0$ is a Hermitian metric on $M$.
Sometimes we are going to simply denote:
\begin{equation*}
F(\chi+dd^c\varphi)\equiv f\big(\lambda[\chi+dd^c\varphi]\big).
\end{equation*}
We also denote $d=\frac{1}{2}(\partial+\bar{\partial})$ and $d^c=\frac{\sqrt{-1}}{2}(\bar{\partial}-\partial)$ so that $dd^c=\sqrt{-1}\partial\bar{\partial}$.

Moreover,  we make the following assumptions on $f$ and $\Gamma$.
\begin{assum}\label{a1.1N}
\begin{enumerate}
\item $\Gamma\subset \bR^n$ is an open symmetric convex cone which contains $\{\lambda\in \bR^n:\lambda_i > 0,\,1\le i\le n\}$ and is contained in $\{\lambda\in \bR^n:\sum_i\lambda_i > 0\}$.  $f:\Gamma\rightarrow \bR_+$ is a smooth and symmetric function.
\item $\frac{\partial f}{\partial \lambda_i}>0,\,\,1\le i\le n$,  $f$ is concave and strictly positive on $\Gamma$ and $\lim_{\Gamma\ni \lambda\rightarrow \lambda_0}f(\lambda)=0$ for any $\lambda_0\in \partial \Gamma$.
\item For any $\lambda\in \Gamma$,  one has $\lim_{R\rightarrow +\infty}f(R\lambda)=+\infty$.
\end{enumerate}
\end{assum}
On a closed Hermitian manifold,  one can only hope to solve (\ref{1.1N}) up to a multiplicative factor on the right hand side,  namely we can only hope to solve:
\begin{equation}\label{1.2NN}
f\big(\lambda[\chi+dd^c\varphi]\big)=e^{G(x)+c},\,\,\,\lambda[\chi+dd^c\varphi]\in \Gamma.
\end{equation}
It is not hard to see that given $G\in C^{\infty}(M)$,  there is at most one constant $c$ that makes (\ref{1.2NN}) admit a smooth solution,  and in general,  this constant (if exists) is very difficult to determine from $G$.  
The question whether (\ref{1.2NN}) admits a solution is a delicate one,  and involves the notion of subsolution,  which is instrumental in the study of Hessian equations and will be explained below.  

Using that $f(\lambda_1,\cdots,\lambda_n)$ is strictly monotone increasing in each variable,  we can define:
\begin{equation}\label{1.3N}
\begin{split}
&f_{\infty,i}(\lambda_1,\cdots,\hat{\lambda}_i,\cdots,\lambda_{n})=\lim_{R\rightarrow +\infty}f(\lambda_1,\cdots,\lambda_{i-1},R,\lambda_{i+1},\cdots\lambda_{n}),\\
&f_{\infty}(\lambda_1,\cdots,\lambda_n)=\min_{1\le i\le n}f_{\infty,i}(\lambda_1,\cdots,\hat{\lambda}_i,\cdots,\lambda_n).
\end{split}
\end{equation}
Here $\hat{\lambda}_i$ simply means this variable is skipped.
For each $f_{\infty,i}$,  the limit is either $+\infty$,  or a concave function defined $\Gamma'$,  which is the projection of $\Gamma$ onto $\bR^{n-1}$ where $\lambda_i$ is excluded.  Here $f_{\infty}(\lambda_1,\cdots,\lambda_n)$ would be defined on 
\begin{equation}\label{1.4N}
\Gamma_{\infty}:=\{\lambda\in \bR^n:\text{ for any $i$,  $(\lambda_1,\cdots,\lambda_{i-1},R,\lambda_{i+1},\cdots,\lambda_n)\in \Gamma$ for all sufficiently large $R$}\}.
\end{equation}
We need to make the following additional assumption on $\chi$:

\begin{assum}\label{chi positivity-2}
$\lambda[\chi](x)\in \Gamma$ for any $x\in M$.
\end{assum}

From (\ref{1.2NN}) and that $f$ is strictly monotone increasing in each variable,  we see that the solution $\varphi$ would satisfy:
\begin{equation}\label{1.5N}
f_{\infty}\big(\lambda[\chi+dd^c\varphi]\big)>e^{G(x)+c},\,\,\,\lambda[\chi+dd^c\varphi]\in \Gamma_{\infty}.
\end{equation}
Therefore,  a necessary condition for the solvability of (\ref{1.2NN}) is the existence of a smooth function $\underline{\varphi}$ that satisfies (\ref{1.5N}).  Such $\underline{\varphi}$ is called a strict subsolution.  In the case of Dirichlet problem (either for a smooth bounded domain in Euclidean spaces,  or manifolds with boundary),  it has been shown by Trudinger \cite{T} and Guan \cite{Guan} that the existence of a subsolution is also sufficient for the solvability (\ref{1.2NN}).  

The case with closed manifold is more subtle.
G.  Székelyhidi \cite{S} proved that if there is a smooth function $\underline{\varphi}$ which is a ``$\mathcal{C}$-subsolution" with respect to a smooth right hand side $G(x)$,  and if there is a smooth solution to $f(\lambda[\chi+dd^c\varphi])=e^{G(x)}$,  then one can estimate $||\varphi||_{2,\alpha}$.  Chu-McCleerey \cite{CM} derived full $C^2$ estimate for solutions with degenerate right hand side. However,  such apriori estimate does not immediately translates to solvability of (\ref{1.2NN}).  Indeed,  if one tries to use continuity path,  and deform the right hand side $G$,  then the corresponding constant $c$ will also change.  Because of this,  a $\mathcal{C}$-subsolution in terms of the function $G$ (together with its constant $c$) in the sense defined by Székelyhi \cite{S} does not necessarily translates to a subsolution along the continuity path.

Note that the situation is much simpler in the special case where $f_{\infty}=+\infty$,  so that (\ref{1.5N}) holds unconditionally,  and the value of the constant $c$ becomes irrelevant.  In our previous work \cite{CX},  we assumed that $f$ satisfies a determinant domination condition,  namely:
\begin{equation*}
f(\lambda)\ge c\big(\Pi_{i=1}^n\lambda_i\big)^{\frac{1}{n}},\text{ for some $c>0$ and any $\lambda\in \Gamma_n$}.
\end{equation*}
This in particular implies $f_{\infty}\equiv +\infty$.  

If $f_{\infty}<+\infty$,  which is the case for complex Hessian quotient equation,  it is of critical importance to understand how the constant $c$ is determined through the equation and the right hand side $e^G$.  
This delicate issue was resolved by the recent work of Guo-Song \cite{GS},  where the authors made the observation that the constant $c$ can be characterized as:
\begin{equation}\label{1.6NNN}
e^c=\inf_{u\in \mathcal{E}_{\Gamma,\chi}}\max_Me^{-G(x)}f\big(\lambda[\chi+dd^cu]\big).
\end{equation}
Here $\mathcal{E}_{\Gamma,\chi}$ denotes the class of smooth functions $u$ on $M$ such that $\lambda[\chi+dd^cu](x)\in \Gamma$,  for any $x$ on $M$.  The authors have shown that the existence of a strict subsolution (a smooth function satisfying (\ref{1.5N}) with the constant $c$ given by (\ref{1.6NNN})) is equivalent to the solvability of (\ref{1.2NN}).  More precisely,  they proved:
\begin{thm}
(Guo-Song,  \cite{GS})
Assume that $f$ and $\Gamma$ satisfy Assumption \ref{a1.1N}.  Let $G\in C^{\infty}(M)$.  Define $c$ by (\ref{1.6NNN}).  
Then the following two statements are equivalent:
\begin{enumerate}
\item There exists a smooth solution to (\ref{1.2NN}).
\item There exists $\underline{\varphi}\in C^{\infty}(M)$,  such that:
\begin{equation*}
f_{\infty}\big(\lambda[\chi+dd^c\underline{\varphi}]\big)>e^{G(x)+c},\,\,\,\lambda[\chi+dd^c\underline{\varphi}]\in \Gamma_{\infty}.
\end{equation*}
\end{enumerate}
\end{thm}

The purpose of this work is to develop the viscosity side of the theory.  In literature, there are several works about applying viscosity theory to PDEs, including Eyssidieux-Guedj-Zeriahi's work \cite{EGZ} on complex Monge-Amp\'ere equation on closed K\"ahler manifolds, Lu's work \cite{L} on $\sigma_k$ equation on domains in $\mathbb{C}^n$ or homogeneous Hermitian manifolds, Dinew-Do-To's work \cite{DDT} on Dirichlet problems of complex Hessian type equations on domain in $\mathbb{C}^n$.  In our setting,  we wish to show that the existence of a strict subsolution in the viscosity sense will lead to a viscosity solution to (\ref{1.2NN}).  To be more specific,  let $G\in C(M)$,  and assume that there exists a viscosity solution $\varphi$ to (\ref{1.2NN}) for some constant $c\in \bR$.  It is not hard to see that $\varphi$ will also satisfy the following in the viscosity sense:
\begin{equation}\label{1.7N}
f_{\infty}\big(\lambda[\chi+dd^c\varphi]\big)\ge e^{G(x)+c},\,\,\,\,\lambda[\chi+dd^c\varphi]\in \Gamma_{\infty}.
\end{equation}
The first result we show is that (\ref{1.7N}) is almost sufficient for the existence of a viscosity solution to (\ref{1.2NN}),   provided that we take $c$ to be defined (\ref{1.6NNN}),  and we assume the subsolution to be strict.  More precisely,  we have:
\begin{thm}
Assume that Assumption \ref{a1.1N}, \ref{chi positivity-2} and \ref{a1.2N} hold.  Let $G\in C(M)$.  Define $c$ by (\ref{1.6NNN}).  Assume that there exists $\underline{u}\in C(M)$ which is a strict subsolution in the viscosity sense: for some $\delta_0>0$,
\begin{equation*}
f_{\infty}\big(\lambda[\chi+dd^c\underline{u}]\big)\ge e^{G+c}+\delta_0,\,\,\,
\lambda[\chi+dd^c\underline{u}]\in \Gamma_{\infty}.
\end{equation*}
Then there exists a viscosity solution to:
\begin{equation*}
f\big(\lambda[\chi+dd^cu]\big)=e^{G(x)+c},\,\,\,\lambda[\chi+dd^cu]\in \Gamma.
\end{equation*}
\end{thm}
In the above,  we need the following assumption in addition to Assumption \ref{a1.1N}:
\begin{assum}\label{a1.2N}
Either $f_{\infty}=+\infty$ on $\Gamma_{\infty}$,  or $f_{\infty}<+\infty$ on $\Gamma_{\infty}$ and:
\begin{equation*}
\lim\sup_{\Gamma_{\infty} \ni \lambda\rightarrow \lambda_0}f_{\infty}(\lambda)=0,\,\,\,\text{ for any $\lambda_0\in \partial\Gamma_{\infty}$}.
\end{equation*}
\end{assum}
This assumption is satisfied if $f(\lambda)=\sigma_k^{\frac{1}{k}}(\lambda),\,\,1\le k\le n$ as well as complex quotient Hessian equations where $f(\lambda)=\big(\frac{\sigma_k(\lambda)}{\sigma_l(\lambda)}\big)^{\frac{1}{k-l}},\,\,1\le l<k\le n$,  $\lambda\in \Gamma_k$.  In the first example,  one has $f_{\infty}(\lambda)=+\infty$,  and in the second example,  one has:
\begin{equation*}
\begin{split}
&f_{\infty}(\lambda)=\min_{1\le i\le n}\big(\frac{\sigma_{k-1}(\lambda_1,\cdots,\hat{\lambda}_i,\cdots,\lambda_n)}{\sigma_{l-1}(\lambda_1,\cdots,\hat{\lambda}_i,\cdots,\lambda_n)}\big)^{\frac{1}{k-l}},\\
&\Gamma_{\infty}=\{\lambda\in \bR^n:\text{ for any $1\le i\le n$,  }\sigma_{j}(\lambda_1,\cdots,\hat{\lambda}_i,\cdots,\lambda_n)>0,\,1\le j\le k-1\}.
\end{split}
\end{equation*}

Unfortunately Assumption \ref{a1.2N} does not follow from Assumption \ref{a1.1N}.  We will explain where we need this assumption when we explain our strategy of proof.

Regarding uniqueness,  we proved:
\begin{thm}\label{t1.3}
Assume that Assumption \ref{chi positivity-2} holds. 
\begin{enumerate}
\item  Assume also that $G(x,z)\in C(M\times \bR)$ with $G$ strictly monotone increasing in the $z$ variable.  Then there is at most one viscosity solution to:
\begin{equation*}
f\big(\lambda[\chi+dd^c\varphi]\big)=e^{G(x,\varphi)},\,\,\,\lambda[\chi+dd^c\varphi]\in \Gamma.
\end{equation*}
\item Let $G\in C(M)$,  then there is at most one constant $c\in \bR$,  such that there exists a viscosity solution to 
\begin{equation*}
f\big(\lambda[\chi+dd^c\varphi]\big)=e^{G(x)+c},\,\,\,\lambda[\chi+dd^c\varphi]\in \Gamma.
\end{equation*}
\end{enumerate}
\end{thm}
The second part of Theorem \ref{t1.3} shows that the constant $c$ given by (\ref{1.6NNN}) is the ``canonical" constant that one can hope to solve (\ref{1.2NN}),  even in the viscosity setting.  Of course,  whether one can actually solve (\ref{1.2NN}) depends on the availability of a subsolution (whether it being smooth or merely in the viscosity sense).  On the other hand,  we don't know whether one has the uniqueness of viscosity solution to (\ref{1.2NN}) in general.  At this moment,  we can only get uniqueness of viscosity solutions contingent on the knowledge of strict monotonicity of the constant $c$ in terms of the right hand side.  The question that is still not fully understood is the following:
\begin{conj}\label{conj1.3}
Let $G\in C(M)$,  we define $c(G)$ using (\ref{1.6NNN}).  Assume that $G'\in C(M)$,  $G'\le G$ and $c(G')=c(G)$.  Does it follow that $G'=G$?
\end{conj}
We prove that an affirmative answer to Conjecture \ref{conj1.3} will imply uniqueness of viscosity solutions.

Now we explain our strategies of proof.
For existence,  we use an approximation argument.  Namely,  we choose $G_j\in C^{\infty}(M)$ such that $G_j\rightarrow G$ uniformly,  and the first step is to solve:
\begin{equation}\label{1.8New}
f\big(\lambda[\chi+dd^cu_j]\big)=e^{G_j+c_j},\,\,\lambda[\chi+dd^cu_j]\in \Gamma.
\end{equation}
In order to solve (\ref{1.8New}),  we need a strict subsolution,  namely a function $\tilde{u}\in C^{\infty}(M)$,  such that for some $\tilde{\delta}_0>0$:
\begin{equation}\label{1.9N}
f_{\infty}\big(\lambda[\chi+dd^c\tilde{u}]\big)\ge e^{G_j+c_j}+\tilde{\delta}_0,\,\,\,\lambda[\chi+dd^c\tilde{u}]\in \Gamma_{\infty}.
\end{equation}
Here $c_j$ is given by (\ref{1.6NNN}) with $G_j$ replacing $G$ there.
In order to get the smooth strict subsolution,  we apply the Richberg technique to regularize $\underline{u}$,  which is the strict subsolution in the viscosity sense.  When applying the Richberg technique,  we need to construct smooth subsolutions locally,  which will be achieved by solving the following Dirichlet problem (very roughly speaking):
\begin{equation}
f_{\infty}\big(\lambda[\chi+dd^cv]\big)=e^{G+c}-\eps,\,\,\lambda[\chi+dd^cv]\in \Gamma_{\infty}\text{ in $B_r(x_0)$},\,\,v=\underline{u}-\eps\text{ on $\partial B_r(x_0)$}.
\end{equation}
Here $B_r(x_0)$ is a ball centered at $x_0$ with small enough radius.
In order to solve the Dirichlet problem,  we will need a non-degeneracy condition,  namely,  we need (with small enough $\eps$):
\begin{equation}\label{1.11}
\inf_{\bar{B}_r(x_0)}\big(e^{G+c}-\eps\big)>\sup_{\partial \Gamma_{\infty}}f_{\infty}.
\end{equation}
The Assumption \ref{a1.2N} is needed exactly to guarantee (\ref{1.11}).  After we are done with constructing the smooth strict subsolution locally,  we can used the regularized maximum to patch the local subsolutions into a global one.

Our next goal is to extract a subsequence of $u_j$ that converges uniformly.  It will be relatively easy to see that the sequence $\{u_j\}$ is precompact in $L^1$.  Then a key step is to improve the $L^1$ convergence to $L^{\infty}$ convergence,  for which we need a stability estimate,  whose proof is a variant of the PDE proof of $L^{\infty}$ estimate for complex Monge-Ampère equations by Guo-Phong-Tong \cite{GPT}.  Related $L^{\infty}$ estimate is obtained for complex Hessian quotient equations on compact K\"ahler manifolds by Sui-Sun in \cite{SS}.

The proof of uniqueness is very similar to our previous work \cite{CX} where we additionally assumed that:
\begin{equation*}
f(\lambda)\ge c\big(\Pi_{i=1}^n\lambda_i\big)^{\frac{1}{n}},\,\,\,\lambda\in \Gamma_n.
\end{equation*}
We will sketch the proof and highlight the difference from our previous proof.  

Now we explain the organization of this paper.  

In Section 2,  we explain some preliminary definitions,  as well as discuss some basic properties on $f$,  as well as $f_{\infty}$.

In Section 3,   we use the Richberg technique to regularize the strict subsolution in the viscosity sense to obtain a smooth strict subsolution,  following the work of Harvey-Lawson-Pliś \cite{HLP}.

In Section 4,  we prove the stability estimate that upgrades the $L^1$ convergence to uniform convergence.

In Section 5,  we discuss the uniqueness issues.

\section{Preliminaries}
\subsection{Some basic definitions}
First we explain what is meant when we say ``touch from above/below". 
\begin{defn}\label{touch}
Let $\varphi$ be a function defined on $M$ and $x_0\in M$.  Let $\psi$ be another function defined on an open subset of $M$ containing $x_0$.  
\begin{enumerate}
\item We say that $\psi$ touches $\varphi$ from above at $x_0$,  if there exists an open neighborhood $U$ of $x_0$ such that $\psi(x_0)=\varphi(x_0)$ and $\psi\ge\varphi$ on $U$,
\item We say that $\psi$ touches $\varphi$ from below at $x_0$,  if there exists an open neighborhood $U$ of $x_0$ such that $\psi(x_0)=\varphi(x_0)$ and $\psi\le\varphi$ on $U$.
\end{enumerate}
\end{defn}

Now we define the notion of viscosity solution for (\ref{1.2NN}):
\begin{defn}
Let $G\in C(M)$ and $c\in \bR$.  Let $\varphi\in C(M)$.  We say that $\varphi$ is a viscosity solution to (\ref{1.2NN}) if the following hold:
\begin{enumerate}
\item For any $x_0\in M$ and any $C^2$ function $P$ defined in a neighborhood of $x_0$ that touches $\varphi$ from above at $x_0$,  one has
\begin{equation*}
\lambda[\chi+dd^cP](x_0)\in \Gamma,\,\,f\big(\lambda[\chi+dd^cP]\big)(x_0)\ge e^{G(x_0)+c}.
\end{equation*}
\item For any $x_0\in M$ and any $C^2$ function $P$ defined in a neighborhood of $x_0$ that touches $\varphi$ from below at $x_0$,  one has: either $\lambda[\chi+dd^cP](x_0)\notin \Gamma$,  or $\lambda[\chi+dd^cP](x_0)\in \Gamma$ and $f\big(\lambda[\chi+dd^cP]\big)(x_0)\le e^{G(x_0)+c}.$
\end{enumerate}
\end{defn}
In the same spirit,  we can generalize the notion of strict subsolution from classical smooth sense to viscosity sense:
\begin{defn}\label{d2.2}
Let $\underline{u}\in C(M),\,G\in C(M)$ and $c\in \bR$.  We say that $\underline{u}$ is a strict subsolution with respect to $(G,c)$ in the viscosity sense if there exists $\delta_0>0$ such that for any $x_0\in M$,  any $C^2$ function $P$ defined in a neighborhood of $x_0$ that touches $\varphi$ from above at $x_0$ one has:
\begin{equation*}
\lambda\big[\chi+dd^cP](x_0)\in \Gamma_{\infty},\,\,f_{\infty}\big(\lambda[\chi+dd^cP]\big)(x_0)\ge e^{G(x_0)+c}+\delta_0.
\end{equation*}
Here $f_{\infty}$ and $\Gamma_{\infty}$ are given by (\ref{1.3N}) and (\ref{1.4N}) respectively.
\end{defn}

Next we need to make precise the notion of strict convex used in this paper.
\begin{defn}\label{d2.4}
Let $\Omega\subset \bR^n$ be an open convex set.  Let $u$ be a function defined on $\Omega$.  We say $u$ is strictly convex,  if for any $x_0\in \Omega$,  there exists an affine function $l_{x_0}$,  such that $u(x)\ge l_{x_0}(x)$ for any $x\in \Omega$,  with $``="$ holds if and only if $x=x_0$.  We say that $u$ is strictly concave if $-u$ is strictly convex.  
\end{defn}
It is clear that if $u\in C^2(\Omega)$ and $D^2u(x)>0$ for $x\in \Omega$,  then $u$ is strictly convex.  It is also clear that if $u_1$ is convex on $\Omega$,  $u_2$ is strictly convex,  then $u_1+u_2$ will be strictly convex. 

Finally,  we make clear the notion of punctually second order differentiable.  
\begin{defn}
Let $\varphi$ be a function on $M$ and $x_0\in M$.  We say that $\varphi$ is punctually second order differentiable at $x_0$ if there is a neighborhood $U$ of $x_0$ and a $C^2$ function $\psi$ defined on $U$ such that 
\begin{equation*}
\lim_{r\rightarrow 0}r^{-2}\sup_{x\in B_r(x_0)}|\varphi(x)-\psi(x)|=0.
\end{equation*}
Here $B_r(x_0)$ denotes the geodesic ball with radius $r$(under the metric $\omega_0$).
\end{defn}

\subsection{Some properties of the operators $f$,  $f_{\infty}$}
First we observe that:
\begin{lem}\label{lem 2.3}
Let $f$ be a nonnegative concave function defined on a closed convex cone $\Gamma$.  Then for any $\lambda\in \Gamma$,  $(0,+\infty)\ni t\mapsto \frac{f(t\lambda)}{t}$ is monotone decreasing.  In particular,  $f(\lambda)\le f(t\lambda)\le tf(\lambda),\,\,t\ge 1$,  and $tf(\lambda)\le f(t\lambda)\le f(\lambda)$ for $0<t\le 1$.
\end{lem}
\begin{proof}
Let $g(t)=\frac{f(t\lambda)}{t}$,  then we find that: $g'(t)=\frac{\sum_i\lambda_i\frac{\partial f}{\partial \lambda_i}(t\lambda)\cdot t-f(t\lambda)}{t^2}.$
On the other hand,  using the concavity of $f$, we see that for any $t_0>0$:
\begin{equation*}
0\le f(0)\le f(t_0\lambda)+\frac{d}{dt}|_{t=t_0}(f(t\lambda))\cdot (-t_0)=f(t_0\lambda)-t_0\sum_i\lambda_i\frac{\partial f}{\partial \lambda_i}(t_0\lambda).
\end{equation*}
To see that $f(t\lambda)\ge f(\lambda)$ for $t\ge 1$,  we observe that $t\mapsto f(t\lambda)$ is monotone increasing.  Indeed,  this function is concave,  and if its derivative is $<0$ for some $t_0>0$,  then it will tends to $-\infty$ as $t\rightarrow +\infty$,  contradicting that $f\ge 0$.
\end{proof}

Next we observe some elementary properties about $f_{\infty}$ and $\Gamma_{\infty}$.
\begin{lem}\label{l2.4N}
Assume that $f$ and $\Gamma$ satisfies Assumption \ref{a1.1N}.  
Define $f_{\infty,i}$ by (\ref{1.3N}).  We also define 
\begin{equation*}
\Gamma_i'=\{(\lambda_1,\cdots,\hat{\lambda}_i,\cdots,\lambda_n)\in \bR^{n-1}:\text{$(\lambda_1,\cdots,\lambda_{i-1},R,\lambda_{i+1},\cdots,\lambda_n)\in \Gamma$ for all $R$ large enough}\}.
\end{equation*}
Here $\hat{\lambda}_i$ means $\lambda_i$ is skipped.  Then:
\begin{enumerate}
\item $\Gamma_i'$ is an open convex cone which is symmetric in the variables $\{\lambda_j\}_{j\neq i}$,  and $\Gamma_{i_1}'=\Gamma_{i_2}'$.  So we may denote them simply as $\Gamma'$.
\item $\Gamma_i'$ contains $\{(\lambda_1,\cdots,\hat{\lambda}_i,\cdots,\lambda_n):\lambda_j > 0,\,1\le j\le n,\,j\neq i\}$.
\item $f_{\infty,i}(\lambda')$ is either identically $+\infty$ or is a concave function on $\Gamma_i'$.  Moreover,  the convergence is uniform on every compact subset of $\Gamma_i'$.
\end{enumerate}
\end{lem}
\begin{proof}
To see (1),  we observe that $(\lambda_1,\cdots,\lambda_{i-1},R,\lambda_{i+1},\cdots,\lambda_n)\in \Gamma$ for all $R$ large enough is actually equivalent to $(\lambda_1,\cdots,\lambda_{i-1},R_0,\lambda_{i+1},\cdots,\lambda_n)\in \Gamma$ for some $R_0\in \bR$.  Indeed,  since $\Gamma_n\subset \Gamma$,  we see that if for some $R_0\in \bR$,  $(\lambda_1,\cdots,\lambda_{i-1},R_0,\lambda_{i+1},\cdots,\lambda_n)\in \Gamma$,  then 
\begin{equation*}
(\lambda_1,\cdots,\lambda_{i-1},R_0,\lambda_{i+1},\cdots,\lambda_n)+te_i\in \Gamma,\,\,t>0.
\end{equation*}
This observation suggests that $\Gamma_i'$ is actually the projection of $\Gamma$ into $\bR^{n-1}$ under the map $\bR^n\ni \lambda\mapsto (\lambda_j)_{j\neq i}$.  Since $\Gamma$ is open,  convex,  and symmetric in $\lambda_i$,  we see that $\Gamma_i'$ is also open,  convex,  and symmetric in $(\lambda_j)_{j\neq i}$ and also $\Gamma_{i_1}'=\Gamma_{i_2}'$.

(2) also follows easily,  by considering the projection of $\Gamma_n$.

To prove (3),  we observe that for each $R>0$,  \sloppy $f_{R,i}(\lambda_1,\cdots,\hat{\lambda}_i,\cdots,\lambda_n):=f(\lambda_1,\cdots,\lambda_{i-1},R,\lambda_{i+1},\cdots,\lambda_n)$ is a concave function defined on $\Gamma_{R,i}'=\{(\lambda_1,\cdots,\hat{\lambda}_i,\cdots,\lambda_n)\in \bR^{n-1}:(\lambda_1,\cdots,\lambda_{i-1},R,\lambda_{i+1},\cdots,\lambda_n)\in \Gamma\}.$ Each $\Gamma_{R,i}'$ is convex,  and is expanding as $R$ increases.  Moreover,  $f_{R,i}$ is mononone increasing in $R$.  Hence if we take the limit as $R\rightarrow \infty$,  we see the limit is either $+\infty$ or is a finite concave function.  Moreover,  the convergence is uniform on compact subsets of $\Gamma_i'$.
\end{proof}
As a consequence,  we see that:
\begin{lem}\label{l2.5}
Assume that $f$ and $\Gamma$ satisfies Assumption \ref{a1.1N}.  Let $f_{\infty}$ and $\Gamma_{\infty}$ be defined through (\ref{1.3N}) and (\ref{1.4N}).  Then one has:
\begin{enumerate}
\item $\Gamma\subset \Gamma_{\infty}$.  
\item $\Gamma_{\infty}$ is an open convex symmetric cone which contains $\Gamma_n$. 
\item Either $f_{\infty}=+\infty$ on $\Gamma_{\infty}$ or $f_{\infty}$ is a concave function on $\Gamma_{\infty}$.
\end{enumerate}
If $f_{\infty}<+\infty$,  then:
\begin{enumerate}
\item $f_{\infty}$ is concave,  symmetric and monotone increasing in every variable.
\item $f_{\infty}>0$ on $\Gamma_{\infty}$.
\end{enumerate}
\end{lem}
\begin{proof}
First we prove $\Gamma\subset \Gamma_{\infty}$.  This follows from $\Gamma_n\subset \Gamma$.  Indeed,  if $\lambda\in \Gamma$,  since $te_i\in \bar{\Gamma}_n$,  then $\lambda+te_i\in \Gamma$,  for any $1\le i\le n$ and $t>0$,  which implies $\lambda\in \Gamma_{\infty}$.

To prove (2),  we simply note that with $\Gamma_{\infty}$ defined by (\ref{1.4N}),  we have:
\begin{equation*}
\Gamma_{\infty}:=\{\lambda \in \bR^n:(\lambda_1,\cdots,\hat{\lambda}_i,\cdots,\lambda_n)\in \Gamma',\,\,1\le i\le n\},
\end{equation*}
where $\Gamma'$ is the projection of $\Gamma$ into $\bR^{n-1}$ as defined in Lemma \ref{l2.4N},  item (1).
From this,  we immediately see that $\Gamma_{\infty}$ is open,  convex and symmetric.  Also we see that it contains $\Gamma_n$ because $\Gamma$ does.

To see item (3),  we note that since 
\begin{equation*}
f_{\infty}(\lambda)=\min_{1\le i\le n}f_{\infty,i}(\lambda_1,\cdots,\hat{\lambda}_i,\cdots,\lambda_n),
\end{equation*}
if for all $i$,  one has $f_{\infty,i}(\lambda')=+\infty$ for $\lambda'\in \Gamma'$,  then we see $f_{\infty}=+\infty$.  If for some $i$,  $f_{\infty,i}$ is a finite function,  then we get $f_{\infty}$ is a finite concave function on $\Gamma_{\infty}$.

Assuming that $f_{\infty}<+\infty$.  It is clearly symmetric,  using the symmetry of each $f_{\infty,i}$.  Also it is concave,  being the minimum of a family of concave functions.  It is monotone increasing in each variable,  since each $f_{\infty,i}$ has the same property.

To see $f_{\infty}>0$ on $\Gamma_{\infty}$,  first we observe that it is clear $f_{\infty}\ge 0$ on $\Gamma_{\infty}$.  Since $f_{\infty}$ is concave,  we can conclude that $f_{\infty}\equiv 0$ if $f_{\infty}(\lambda)=0$ for some $\lambda\in \Gamma_{\infty}$.  This is clearly impossible,  because $f_{\infty}(\lambda)\ge f(\lambda)>0$ for $\lambda \in \Gamma$.
\end{proof}
Later we are going to need:
\begin{lem}\label{l2.9New}
Assume that $f$ and $\Gamma$ satisfies Assumption \ref{a1.1N}.  
Let $K\subset \Gamma_{\infty}$ be a compact subset.
\begin{enumerate}
\item There exists $R_0>0$ sufficiently large,  such that for any $\lambda\in K$ and any $\tilde{\lambda}\in \bar{\Gamma}_n$ with $|\tilde{\lambda}|\ge R_0$,  one has $\lambda+\tilde{\lambda}\in \Gamma$.
\item If $f_{\infty}=+\infty$,  then for any $B>0$,  there exists $R_0>0$ sufficiently large,  such that for any $\lambda\in K$ and any $\tilde{\lambda}\in \bar{\Gamma}_n$ with $|\tilde{\lambda}|\ge R_0$,  one has $f(\lambda+\lambda')\ge B$.
\item If $f_{\infty}<+\infty$,  then for any $\eps>0$,  there exists $R_0>0$ such that for any $\lambda\in K$ and any $\tilde{\lambda}\in \bar{\Gamma}_n$ with $|\tilde{\lambda}|\ge R_0$,  one has $f(\lambda+\tilde{\lambda})\ge f_{\infty}(\lambda)-\eps$.
\end{enumerate}
\end{lem}
\begin{proof}
First we prove the three statements when $K=\{\lambda_*\}$ for some $\lambda_*\in \Gamma_{\infty}$.  If $\tilde{\lambda}\in \Gamma_n$,  and $|\tilde{\lambda}|\ge R_0$,  we may assume that $\tilde{\lambda}_1\ge \frac{R_0}{\sqrt{n}}$ without loss of generality.  Therefore,  with $R_0$ sufficiently large,  we have $\lambda_*+\frac{R_0}{\sqrt{n}}e_1\in \Gamma$.  Moreover,  $\lambda_*+\tilde{\lambda}=\lambda_*+\frac{R_0}{\sqrt{n}}e_1+\hat{\lambda},\,\,\hat{\lambda}\in \Gamma_n$.  Therefore,  we see that $\lambda_*+\tilde{\lambda}\in \Gamma$.

If $f_{\infty}=+\infty$,  then one has:
\begin{equation*}
\lim_{R\rightarrow \infty}f(R,\lambda_2,\cdots,\lambda_n)=+\infty.
\end{equation*}
Hence with $R_0$ sufficiently large,  one has $f(\lambda_*+\frac{R_0}{\sqrt{n}}e_1)\ge B$.  Using the monotonicity of $f$,  we get:
\begin{equation*}
f(\lambda_*+\tilde{\lambda})\ge f(\lambda_*+\frac{R_0}{\sqrt{n}}e_1)\ge B.
\end{equation*}

If $f_{\infty}<+\infty$,  then one has:
\begin{equation*}
\lim_{R\rightarrow \infty}f(R,\lambda_{2*},\cdots,\lambda_{n*})\ge f_{\infty}(\lambda_*).
\end{equation*}
Hence with $R_0$ sufficiently large,  one has $f\big(\lambda_*+\frac{R_0}{\sqrt{n}}e_1)\ge f_{\infty}(\lambda_*)-\eps$.  Hence $f(\lambda_*+\tilde{\lambda})\ge f(\lambda_*+\frac{R_0}{\sqrt{n}}e_1)\ge f(\lambda_*)-\eps$.

Next if $K$ is a polytope,  meaning that it is the closed convex hull of a finite set,  we can choose $R_0$ sufficiently large so that (1)-(3) holds for its vertices.  If $\lambda\in K$ and $\tilde{\lambda}\in \Gamma_n$ with $|\tilde{\lambda}|\ge R_0$,  we see that:
\begin{equation*}
\lambda+\tilde{\lambda}=\sum_l\mu_l\lambda^{(l)}+\tilde{\lambda}=\sum_l\mu_l(\lambda^{(l)}+\tilde{\lambda}),\,\,\,\mu_l\ge 0,\,\,\,\sum_l\mu_l=1.
\end{equation*}
In the above,  we have represented $\lambda$ as a convex combination of its vertices.  Now we know that each $\lambda^{(l)}+\tilde{\lambda}\in \Gamma$,  we see $\lambda+\tilde{\lambda}\in \Gamma$,  because of convexity.

The argument for (2) and (3) is similar,  using that $f$ is a concave function.

If $K$ is a general compact set,  then it can be covered by finitely many open polytopes compactly contained in $ \Gamma_{\infty}$.
\end{proof}

We also observe that:
\begin{lem}\label{l2.7}
Assume that $f$ and $\Gamma$ satisfies Assumption \ref{a1.1N}.  
If $f_{\infty}<+\infty$,  then $\lim_{R\rightarrow +\infty}f_{\infty}(R\lambda)=+\infty$.  Moreover,  this convergence is uniform on any compact set $K\subset \Gamma_{\infty}$.
\end{lem}
\begin{proof}
First we observe that one has:
\begin{equation}\label{2.1N}
\sum_i\lambda_i\frac{\partial f_{\infty}}{\partial \lambda_i}(\lambda)\ge 0,\,\,\,\lambda\in \Gamma_{\infty}.
\end{equation}
Indeed,  for any $R>0$,  one concludes from concavity:
\begin{equation*}
0\le f(R\lambda)\le f(\lambda)+\sum_{i=1}^n\lambda_i\frac{\partial f_{\infty}}{\partial \lambda_i}(\lambda)\cdot (R-1).
\end{equation*}
Therefore,  $\sum_{i=1}^n\lambda_i\frac{\partial f_{\infty}}{\partial \lambda_i}(\lambda)\ge \frac{-f(\lambda)}{R-1}.$  Letting $R\rightarrow \infty$,  one gets (\ref{2.1N}).  
Now we define $g_R(\lambda)=f_{\infty}(R\lambda)$.  Then each $g_R(\lambda)$ is a concave function on $\Gamma_{\infty}$ and it is monotone increasing in $R$,  due to (\ref{2.1N}).  Therefore,  $g_R(\lambda)$ either converges to $+\infty$ everywhere on $\Gamma_{\infty}$ or converge to a finite concave function on $\Gamma_{\infty}$.  Moreover,  this convergence is uniform on compact subsets of $\Gamma_{\infty}$.  On the other hand,  for $\lambda\in \Gamma$,  one has $f_{\infty}(R\lambda)\ge f(R\lambda)$ and the latter tends to $+\infty$ as $R\rightarrow +\infty$.  Hence we must have $\lim_{R\rightarrow \infty}f_{\infty}(R\lambda)=+\infty$,  uniform on compact subsets of $\Gamma_{\infty}$.
\end{proof}

One more thing we need is that:
\begin{lem}\label{l2.8}
Assume that $f$ and $\Gamma$ satisfies Assumption \ref{a1.1N}.  Let $f_{\infty}$ and $\Gamma_{\infty}$ be defined by (\ref{1.3N}) and (\ref{1.4N}).  Then:
\begin{enumerate}
\item Either $\Gamma_{\infty}=\bR^n$ or $\Gamma_{\infty}\subset \Gamma_1:=\{\lambda\in \bR^n:\sum_i\lambda_i > 0\}$.
\item If $\Gamma_{\infty}=\bR^n$,  then one must have $f_{\infty}=+\infty$.
\end{enumerate}
\end{lem}
\begin{proof}
First we prove (1).  Assume that $\Gamma_{\infty}\neq \bR^n$,  we just need to show that $\Gamma_{\infty}\subset \bar{\Gamma}_1$.  Assume that $\Gamma_{\infty}$ is not contained in $\Gamma_1$.  That is,  there exists $\lambda_0\in \Gamma_{\infty}$,  such that $\sum_{i}\lambda_{0,i}<0$.  Using that $\Gamma_{\infty}$ is symmetric in its variables,  we see that for any permutation $\sigma\in S_n$,  one has $\lambda_{\sigma,0}\in \Gamma_{\infty}$,  where $\lambda_{\sigma,0}=(\lambda_{0,\sigma(1)},\cdots,\lambda_{0,\sigma(n)})$.
Using the convexity of $\Gamma_{\infty}$,  we see that $\frac{1}{n!}\sum_{\sigma\in S_n}\lambda_{\sigma,0}\in \Gamma_{\infty}$,  which is a multiple of $(-1,\cdots,-1)$.  Hence $\Gamma_{\infty}$ contains $\{(-t,\cdots,-t):t>0\}$.  On the other hand,  we know that $\Gamma_n\subset \Gamma_{\infty}$ and that:
\begin{equation*}
\{(-t,\cdots,-t):t\ge 0\}+\Gamma_n=\bR^n.
\end{equation*}
This proves $\Gamma_{\infty}=\bR^n$.

Now we prove (2).  Assume that $f_{\infty}<\infty$,  we know from Lemma \ref{l2.5} that $f_{\infty}>0$ on $\bR^n$.  For any $\lambda \in \bR^n$,  we have:
\begin{equation*}
f_{\infty}(\lambda_1,\cdots,\lambda_{n-1},R)=f_{\infty}\big(R(\frac{\lambda_1}{R},\cdots,\frac{\lambda_{n-1}}{R},1)\big)
\end{equation*}
Note that for $R\ge 1$ is chosen sufficiently large,  one has $(\frac{\lambda_1}{R},\cdots,\frac{\lambda_{n-1}}{R},1)$ remains in a compact set (of $\Gamma_{\infty}$ which is $\bR^n$).  Therefore,  we may use Lemma \ref{l2.7} to conclude that $f_{\infty}(\lambda_1,\cdots,\lambda_{n-1},R)\rightarrow +\infty$ as $R\rightarrow +\infty$.  That is,  $f_{\infty,n}(\lambda_1,\cdots,\lambda_{n-1})=+\infty$,  for any $(\lambda_1,\cdots,\lambda_{n-1})\in \bR^{n-1}$.  Similarly,  one has $f_{\infty,i}(\lambda_1,\cdots,\hat{\lambda}_i,\cdots,\lambda_n)=+\infty$.  This proves $f_{\infty}=+\infty$.
\end{proof}

The following lemma should be known to the experts:
\begin{lem}\label{lem 2.10}
Exactly one of the following two statements hold for $\Gamma_{\infty}$:
\begin{enumerate}
\item $\partial \Gamma_{\infty}$ contains $\{\lambda e_i: \lambda \ge 0\}$ for any $i$.
\item $\Gamma_{\infty}= \mathbb{R}^n$
\end{enumerate}
\end{lem}
\begin{proof}
Since $\Gamma_n \subset \Gamma \subset \Gamma_{\infty}$, we have that $(0,0,...,0,1)\in \overline{\Gamma_{\infty}}$. As a result, if (1) doesn't hold, we have that $e_i \in \Gamma_{\infty}$ for any $i$. Then there exists a small positive constant $\eps$ such that $-\eps e_1 + e_i \in \Gamma_{\infty}$ for any $i \ge 2$. Then, using the definition of $\Gamma_{\infty}$, we have that $-\eps e_1 \in \Gamma_{\infty}$. Using the symmetry of $\Gamma_{\infty}$, we have that $- \eps e_i \in \Gamma_{\infty}$ for any $i$. Combining this fact and the fact that $e_i \in \Gamma_{\infty}$ and $\Gamma_{\infty}$ is a convex cone, we get that $\Gamma_{\infty}= \mathbb{R}^n$.
\end{proof}

\begin{prop}\label{prop full measure}
Suppose that $\Gamma_{\infty}\neq \mathbb{R}^n$. Then for any $k$, we have that 
\begin{equation*}\label{e 3.1}
\Gamma_{\infty} \subset \{\lambda \in \mathbb{R}^n: \sum_{i\neq k}\lambda_i > 0\}.
\end{equation*}
\end{prop}

\begin{proof}
According to Lemma \ref{lem 2.10}, we have that $e_n \in \partial \Gamma_{\infty}.$ Using Lemma \ref{l2.14N} below,  we may assume that $\Sigma_v:= \{\lambda \in \mathbb{R}^n: (\lambda, v)=0\}$ be a supporting plane of $\Gamma_{\infty}$ at $e_n$ with $\Gamma_{\infty}\subset \{(\lambda,v)\ge 0\}$ and $|v|=1$.
Denote $v=(v_1,...,v_n)$. Since both $e_n$ and $0$ are on $\Sigma_v$, we have that $v_n=0$. Since $\Gamma_n \subset \Gamma_{\infty}$, we have that $v_i \ge 0$ for any $i$,  by taking $\lambda=e_i$.  Using the symmetry of $\Gamma_{\infty}$, we have that the following permutation of the components of $v$: $(v_{n-1},v_1,...,v_{n-2},0)$, $(v_{n-2}, v_{n-1},v_1,...,v_{n-3},0)$, ..., $(v_2,v_3,...,v_{n-1},v_1,0)$ each defines a supporting plane of $\Gamma_{\infty}$ at $(0,...,0,1)$.  Since the space of normal vectors of supporting planes at a point is convex, we have that the mean of the vectors $v$, $(v_{n-1},v_1,...,v_{n-2},0)$, $(v_{n-2}, v_{n-1},v_1,...,v_{n-3},0)$, ..., $(v_2,v_3,...,v_{n-1},v_1,0)$ which is $\frac{\sum_{i=1}^{n-1}v_i}{n-1}(1,1,...,1,0)$ defines a supporting plane of $\Gamma_{\infty}$ at $(0,0,...,0,1)$. Since $v_i \ge 0$ and there exists $v_k$ that is positive. we have that $\frac{\sum_{i=1}^{n-1}v_i}{n-1} >0$. Then we have that  $(1,1,...,1,0)$ defines a supporting plane of $\Gamma_{\infty}$ at $(0,0,...,0,1)$. This concludes the proof of the (\ref{e 3.1}) for $k=n$. Using the symmetry of $\Gamma_{\infty}$ we see that (\ref{e 3.1}) holds for any $k$.
\end{proof}
In the above proof,  we used the following lemma:
\begin{lem}\label{l2.14N}
Let $\Gamma$ be an open convex cone in $\bR^n$.  Let $\Sigma$ be a supporting plane of $\Gamma$ at $\lambda_0\in \partial \Gamma$,  namely there exists $v\in \bR^n$, $|v|=1$ and $c\in \bR$,  such that $(\lambda,v)\ge c,\,\,\forall \lambda\in \Gamma$,  and $(\lambda_0,v)=c$.  Then one must have $c=0$.
\end{lem}
\begin{proof}
It is clear that $(\lambda,v)\ge c$ will hold for $\lambda\in \bar{\Gamma}$.  Since $\lambda_0\in \partial \Gamma$,  we also have $t\lambda_0\in \partial \Gamma$ for any $t>0$,  hence we get $tc\ge c$ for any $t>0$.  This implies $c=0$.
\end{proof}

\section{Regularization of subsolution}
The goal of this section is to show that a strict subsolution in the viscosity sense can be regularized into a smooth strict subsolution.  This is necessary only if $f_{\infty}<+\infty$,  so that we must have $\Gamma_{\infty}\subset \Gamma_1$ by Lemma \ref{l2.8}.   So we are going to assume the following holds throughout this section:
\begin{equation}
f_{\infty}(\lambda)<+\infty,\,\,\,\lambda\in \Gamma_{\infty},\,\,\,\Gamma_{\infty}\subset \Gamma_1.
\end{equation}
More precisely,  we wish to show that:
\begin{thm}\label{t3.1}
Let $G\in C(M)$,  we define $c$ to be given by (\ref{1.6NNN}).  Let $\underline{u}\in C(M)$ be a strict subsolution with respect to $(G,c)$ in the viscosity sense as defined by Definition \ref{d2.2} for some $\delta_0>0$.  Then for any $0<\eps\le \frac{1}{2}\delta_0$,  there exists $u_{\eps}\in C^{\infty}(M)$ such that: 
\begin{enumerate}
\item $\lambda[\chi+dd^cu_{\eps}]\in \Gamma_{\infty}$,  $f_{\infty}\big(\lambda[\chi+dd^cu_{\eps}]\big)\ge e^{G(x)+c}+\delta_0-\eps$,
\item $||u_{\eps}-\underline{u}||_{L^{\infty}}\le \eps$.
\end{enumerate}
\end{thm}
We are going to construct $u_{\eps}$ via Richberg technique,  following the argument by Harvey-Lawson-Pliś \cite{HLP}.  The details will be carried out in the subsequent sections.  As a preliminary step,  we first show that the above notion of strict subsolution can be slightly strengthened.  More precisely,  
\begin{lem}\label{l3.1}
Let $\underline{u}\in C(M)$ be a strict subsolution with respect to $(G,c)$ in the viscosity sense as defined by Definition \ref{d2.2} for some $\delta_0>0$.  There exists a universal constant $\eps_0'>0$ such that for any $0<\eps\le \frac{1}{2}\delta_0$,  one has
\begin{equation*}
\lambda[\chi-\eps_0'\eps\omega_0+dd^c\frac{\underline{u}}{1+\eps}]\in \Gamma_{\infty},\,\,f_{\infty}\big(\lambda[\chi-\eps_0'\eps\omega_0+dd^c
\frac{\underline{u}}{1+\eps}]\big)\ge e^{G+c}+\delta_0-C_0\eps.
\end{equation*}
Here the choice of $\eps_{0}$ depends on the form $\chi$ and $\omega_0$,  and $C_0$ will additionally depend on an upper bound of $e^{G+c}$.
\end{lem}

The next step will be to regularize the operator $f_{\infty}$ to a smooth operator $g_{\eps_1}$.  More precisely,  we have the following technical result:
\begin{prop}\label{p3.2N}
For any $0<\eps<\frac{1}{2}$,  there exists a function $g_{\eps}(\lambda):\Gamma_{\infty}\rightarrow \bR$ such that the following properties hold:
\begin{enumerate}
\item $g_{\eps}\in C^{\infty}(\Gamma_{\infty})$ and symmetric in $(\lambda_1,\cdots,\lambda_n)\in \Gamma_{\infty}$,
\item $(1-\eps)f_{\infty}\le g_{\eps}\le f_{\infty}+\frac{\pi}{2}\eps$ on $\Gamma_{\infty}$.  
\item $g_{\eps}$ is concave,  $\frac{\partial g_{\eps}}{\partial\lambda_i}>0$ on $\Gamma_{\infty}$ and that $\lim_{R\rightarrow +\infty}g_{\eps}(R\lambda)=+\infty$ for any $\lambda \in \Gamma_{\infty}$.
\end{enumerate}
\end{prop}
The uniform closeness of $g_{\eps_1}$ to $f_{\infty}$ will make sure that a strict subsolution in terms of $f_{\infty}$ will also be a strict subsolution of $g_{\eps_1}$ and vice versa.  The smoothness of $g_{\eps_1}$ will be needed when we construct the strict subsolution locally.  Then we need the Richberg technique to patch the local strict subsolution to a global one on $M$.  More precisely,  we have:
\begin{prop}\label{p3.2}
Let $G\in C^{\infty}(M)$,  we define $c$ to be given by (\ref{1.6NNN}).  Assume that for some $\delta_1>0$,  $\eps_2'>0$,  there exists $u_1\in C(M)$ that satisfies the following in the viscosity sense,  for any $\eps_1$ small enough:
\begin{equation}\label{3.1NN}
\lambda[\chi-\eps_2'\omega_0+dd^cu_1]\in \Gamma_{\infty},\,\,g_{\eps_1}\big(\lambda[\chi-\eps_2'\omega_0+dd^cu_1\big)\ge e^{G(x)+c}+\delta_1.
\end{equation}
Then for such $\eps_1>0$ and any $\eps>0$,  there exists $u_{\eps_1,\eps}\in C^{\infty}(M)$ such that:
\begin{enumerate}
\item $\lambda[\chi+dd^cu_{\eps_1,\eps}]\in \Gamma_{\infty}$ and $g_{\eps_1}\big(\lambda[\chi+dd^cu_{\eps_1,\eps}]\big)\ge e^{G(x)+c}+\delta_1-\eps$,
\item $||u_{\eps_1,\eps}-\underline{u}||_{L^{\infty}}\le \eps$.
\end{enumerate}
\end{prop}

First we explain how to use Lemma \ref{l3.1},  as well as Proposition \ref{p3.2N} and \ref{p3.2} below to prove Theorem \ref{t3.1}.
\begin{proof}
(of Theorem \ref{t3.1}) We fix $0<\eps\le \frac{1}{2}\delta_0$,  First we may use Lemma \ref{l3.1} to get that:
\begin{equation*}
\lambda[\chi-\eps_0'\eps\omega_0+dd^c\frac{\underline{u}}{1+\eps}]\in \Gamma_{\infty},\,\,\,f_{\infty}\big(\lambda[\chi-\eps_0'\eps\omega_0+dd^c\frac{\underline{u}}{1+\eps}]\big)\ge e^{G+c}+\delta_0-C_0\eps.
\end{equation*}
Next we wish to use Proposition \ref{p3.2N} to translate this viscosity condition in terms of $g_{\eps_1}$.  Indeed,  let $x_0\in M$ and $P$ is a $C^2$ function that touches $\frac{\underline{u}}{1+\eps}$ from above at $x_0$,  we then have:
\begin{equation*}
\lambda[\chi-\eps_0'\eps\omega_0+dd^cP](x_0)\in \Gamma_{\infty},\,\,\,f_{\infty}\big(\lambda[\chi-\eps_0'\eps\omega_0+dd^cP]\big)(x_0)\ge e^{G+c}+\delta_0-C_0\eps.
\end{equation*}
Let $\eps_1>0$,  from Proposition \ref{p3.2N} we see that:
\begin{equation*}
g_{\eps_1}\big(\lambda[\chi-\eps_0'\eps\omega_0+dd^cP]\big)\ge (1-\eps_1)\big(e^{G+c}+\delta_0-C_0\eps\big)\ge e^{G+c}+\delta_0-C_1\eps_1-C_0\eps.
\end{equation*}
Here $C_1$ is universal,  and depends only on the upper bound of $e^{G+c}$.
Now we are in a position to apply Proposition \ref{p3.2},  with $\eps_2'=\eps_0\eps$,  $\delta_1=\delta_0-C\eps_1-C_0\eps$.  Also we are taking $u_1=\frac{\underline{u}}{1+\eps}$.  Then we can conclude that for any $\eps'>0$,  one can find a function $\tilde{u}\in C^{\infty}(M)$ such that:
\begin{equation*}
\lambda[\chi+dd^c\tilde{u}]\in \Gamma_{\infty},\,\,\,g_{\eps_1}\big(\lambda[\chi+dd^c\tilde{u}]\big)\ge e^{G+c}+\delta_1-\eps'.
\end{equation*}
Moreover,  $||\tilde{u}-\frac{\underline{u}}{1+\eps}||_{L^{\infty}}\le \eps'$.
Hence,  we have:
\begin{equation*}
\begin{split}
&f_{\infty}\big(\lambda[\chi+dd^c\tilde{u}]\big)\ge g_{\eps_1}\big(\lambda[\chi+dd^c\tilde{u}]\big)-\frac{\pi}{2}\eps_1\ge e^{G+c}+\delta_1-\eps'-\frac{\pi}{2}\eps_1\\
&=e^{G+c}+\delta_0-C'\eps_1-\eps'-C_0'\eps.
\end{split}
\end{equation*}
Also we have 
\begin{equation*}
||\tilde{u}-\underline{u}||_{L^{\infty}}\le ||\tilde{u}-\frac{\underline{u}}{1+\eps}||_{L^{\infty}}+||\frac{\underline{u}}{1+\eps}-\underline{u}||_{L^{\infty}}\le \eps'+||\underline{u}||_{L^{\infty}}\eps.
\end{equation*}
 Therefore,  we just need to choose $\eps_1>0,\,\eps'>0$ small enough,  such that $C'\eps_1<\eps,\,\eps'<\eps$,  which gives us $f_{\infty}\big(\lambda[\chi+dd^c\tilde{u}]\big)\ge e^{G+c}+\delta_0-C_0''\eps.$
\end{proof}
Now we prove Lemma \ref{l3.1}.  Proposition \ref{p3.2} and \ref{p3.2N} will be proved in the subsequent subsections.
\begin{proof}
(of Lemma \ref{l3.1})
Since $\lambda[\chi] \in \Gamma_{\infty}$ and $\chi$ is smooth, we have that there exists a small constant $\eps_0>0$ such that $\lambda[\chi-\eps_0 \omega]\in \Gamma_{\infty}$. 
Now let $x_0\in M$ and $P$ is a $C^2$ function defined in a neighborhood of $x_0$ that touches $\underline{u}$ from above at $x_0$. 
We then have:
\begin{equation*}
\lambda[\chi+dd^cP](x_0)\in \Gamma_{\infty},\,\,f_{\infty}\big(\lambda[\chi+dd^cP]\big)(x_0)\ge e^{G+c}+\delta_0.
\end{equation*}
For any $\eps>0$,  since $\lambda[\chi-\eps_0\omega]\in \Gamma$,  we see that $\lambda[\chi+dd^cP+\eps(\chi-\eps_0\omega_0)]\in \Gamma_{\infty}$.

Next we use the concavity of $F_{\infty}$:
\begin{equation}
\begin{split}
&f_{\infty}\big(\lambda[\chi+dd^cP])(x_0)\le f_{\infty}\big(\lambda[\eps(\chi-\eps_0\omega_0)+\chi+dd^cP]\big)(x_0)\\
&+\frac{\partial F_{\infty}}{\partial h_{i\bar{j}}}\big(\lambda[\eps(\chi-\eps_0\omega_0)+\chi+dd^cP]\big)(x_0)\eps\cdot (-\chi_{i\bar{j}}+\eps_0(\omega_0)_{i\bar{j}})\le 0.
\end{split}
\end{equation}
The last inequality follows from the following Lemma \ref{l3.4N} with $A_{i\bar{j}}=\eps(\chi_{i\bar{j}}-\eps_0(\omega_0)_{i\bar{j}})+\chi_{i\bar{j}}+P_{i\bar{j}}$,  $B_{i\bar{j}}=\chi_{i\bar{j}}-\eps_0(\omega_0)_{i\bar{j}}$.
Therefore we get:
\begin{equation*}
(1+\eps)f_{\infty}\big(\lambda[\chi-\frac{\eps\eps_0}{1+\eps}\omega_0+dd^c\frac{P}{1+\eps}]\big)\ge f_{\infty}\big(\lambda[(1+\eps)\chi-\eps\eps_0\omega_0+dd^cP]\big)\ge e^{G+c}+\delta_0.
\end{equation*}
Therefore,
\begin{equation*}
f_{\infty}\big(\lambda[\chi-\frac{\eps\eps_0}{1+\eps}\omega_0+dd^c\frac{P}{1+\eps}]\big)\ge \frac{e^{G+c}+\delta_0}{1+\eps}\ge e^{G+c}+\delta_0-C_0\eps.
\end{equation*}
Here $C_0$ depends on an upper bound of $e^{G+c}$.  Also we may take $\eps_0'=\frac{1}{2}\eps_0$ because $\frac{\eps_0}{1+\eps}\ge \frac{1}{2}\eps_0$ with $\eps$ small enough.
\end{proof}
In the following,  we use the following lemma which can be found in the Lemma 4.7 of \cite{CX}:
\begin{lem}\label{l3.4N}
Let $A$,  $B$ be two Hermitian matrices such that $\lambda(g^{i\bar{k}}A_{j\bar{k}}),\,\lambda(g^{i\bar{k}}B_{j\bar{k}})\in \Gamma$.  Define $F(h)=f\big(\lambda(g^{i\bar{k}}h_{j\bar{k}})\big)$.  Then we have:
\begin{equation*}
\sum_{j,k=1}^n\frac{\partial F}{\partial h_{j\bar{k}}}(A)B_{j\bar{k}}\ge 0.
\end{equation*}
\end{lem}

\subsection{Regularization of $f_{\infty}$}
The goal of this subsection is to prove Proposition \ref{p3.2N}.  It will follow from the following general regularization result of a strictly convex function.
\begin{prop}\label{p3.4}
Let $\Omega\subset \bR^n$ be an open convex subset (possibly unbounded).  Let $u$ be a strictly convex function on $\Omega$.  Let $h\in C(\Omega)$ with $h>0$.  Then there exists a smooth convex function $\tilde{u}$ defined on $\Omega$ with $u\le \tilde{u}\le u+h$ and $D^2\tilde{u}>0$ on $\Omega$.
\end{prop}
Before we prove this proposition,  first we explain how to use it to prove Proposition \ref{p3.2N}.  
\begin{proof}
(of Proposition \ref{p3.2N}) Let $0<\eps<1$.   We can take $\Omega=\Gamma_{\infty}$ and we must modify $-f_{\infty}$ to be strictly convex in order to apply Proposition \ref{p3.4}.  Therefore,  we consider:
\begin{equation}\label{g eps}
\widetilde{f}_{\eps} (\lambda) = f_{\infty}(\lambda) + \frac{ \eps}{n} \sum_{i=1}^n \arctan(\sum_{k\neq i}\lambda_k),
\end{equation}
for any $\lambda \in \overline{\Gamma}_{\infty}.$

\begin{lem}\label{strict convexity}
Let $\widetilde{f}_{\eps}$ be defined by (\ref{g eps}). Then we have that $-\widetilde{f}_{\eps}$ is strictly convex on $\Gamma_{\infty}$ in the sense defined by Definition \ref{d2.4}.
\end{lem}
\begin{proof}
Since we already know that $-f_{\infty}$ is convex, it suffices to show that \sloppy $-\sum_{i=1}^n \arctan(\sum_{k\neq i}\lambda_k)$ is strictly convex. Since $-\sum_{i=1}^n \arctan(\sum_{k\neq i}\lambda_k)$ is smooth, the strict convexity can be follows from the positivity of its Hessian matrix. Denote $v_i =\sum_{k\neq i} e_k.$ Then we have that the Hessian matrix of $-\sum_{i=1}^n \arctan(\sum_{k\neq i}\lambda_k)$ is
\begin{equation*}
\text{Hess}_{pq} = - \sum_{i=1}^n \arctan''(\sum_{k\neq i}\lambda_k) (v_i)_p \cdot (v_i)_q.
\end{equation*}
According to the Proposition \ref{prop full measure}, we have that for $\sum_{k\neq i}\lambda_k >0$ on $\Gamma_{\infty}$. Since $\arctan''(t)<0$ for any $t >0$, we have that $-\arctan''(\sum_{k\neq i}\lambda_k) >0$ for each $i$. Since $\{v_i\}$ is a basis of $\mathbb{R}^n$, we have that $M$ is a positive matrix.  Indeed,  for $\xi\in \bR^n$,  one has:
\begin{equation*}
\begin{split}
&\sum_{p,q=1}^n\text{Hess}_{pq}\xi_p\xi_q=-\sum_{p,q=1}^n\sum_{i=1}^n\arctan''(\sum_{k\neq i}\lambda_k)(v_i)_p\xi_p(v_i)_q\xi_q\\
&=-\sum_{i=1}^n\arctan''(\sum_{k\neq i}\lambda_k)|(v_i,\xi)|^2\ge 0.
\end{split}
\end{equation*}
If it is zero,  we would have $(v_i,\xi)=0$ for all $i$.  Since $\{v_i\}_{1\le i\le n}$ forms a basis,  we see that one must have $\xi=0$.  This proves $\text{Hess}>0$.
\end{proof}

In summary,  now we have that $\tilde{f}_{\eps}$ satisfies:
\begin{enumerate}
\item $\tilde{f}_{\eps}$ is strictly concave,  symmetric,  and $\frac{\partial \tilde{f}_{\eps}}{\partial \lambda_i}>0,\,\,1\le i\le n$,
\item $f_{\infty}\le \tilde{f}_{\eps}\le f_{\infty}+\frac{\pi}{2}\eps$.
\end{enumerate}
In the above $f_{\infty}\le \tilde{f}_{\eps}$ used Proposition \ref{prop full measure}.

Now we wish to apply Proposition \ref{p3.4} to $-\tilde{f}_{\eps}$ which is strictly convex,  and $h=\eps f_{\infty}$,  which is strictly positive on $\Gamma_{\infty}$. 
Therefore,  we get a strictly convex smooth function $-\hat{f}_{\eps}(\lambda)$ defined on $\Gamma_{\infty}$ which satisfies:
\begin{equation*}
-\tilde{f}_{\eps}\le -\hat{f}_{\eps}\le -\tilde{f}_{\eps}+\eps f_{\infty}.
\end{equation*}
Finally,  we need to make the function symmetric in all variables,  so we define:
\begin{equation*}
g_{\eps}(\lambda_1,\cdots,\lambda_n)=\frac{1}{n!}\sum_{\sigma\in S_n}\hat{f}_{\eps}\big(\lambda_{\sigma(1)},\cdots,\lambda_{\sigma(n)}\big),
\end{equation*}
where $S_n$ denotes the group of permutations of $n$ elements.
Now it only remains to check item (1)-(3) in Proposition \ref{p3.2N}.  Item (1) is clear.  To see item (2),  we note that:
\begin{equation}\label{3.4NNew}
\hat{f}_{\eps}\le \tilde{f}_{\eps}\le f_{\infty}+\frac{\pi}{2}\eps,\,\,\,\,\hat{f}_{\eps}\ge \tilde{f}_{\eps}-\eps f_{\infty}\ge (1-\eps)f_{\infty}.
\end{equation}
Since $f_{\infty}$ is symmetric in $(\lambda_1,\cdots,\lambda_n)$,  (\ref{3.4NNew}) would carry over to $g_{\eps}$,  hence we have:
\begin{equation}
(1-\eps)f_{\infty}\le g_{\eps}\le f_{\infty}+\frac{\pi}{2}\eps.
\end{equation}
This verifies (2).

The strict concavity of $g_{\eps}$ follows from the strict concavity of $\hat{f}_{\eps}$ which is built into the construction since we have $D^2\hat{f}_{\eps}<0$.  Therefore we have $D^2g_{\eps}<0$ on $\Gamma_{\infty}$.  It only remains to check $\frac{\partial g_{\eps}}{\partial \lambda_i}>0$.  If not,  then there exists $\lambda_0\in \Gamma_{\infty}$ such that $\frac{\partial g_{\eps}}{\partial \lambda_i}(\lambda_0)\le 0$.  Now we consider the function $k(t):=t\mapsto g_{\eps}(\lambda_0+te_i)$.   We know that $k''(t)<0$ (since $D^2g_{\eps}<0$) and $k'(0) \le 0$.  Hence $k'(t)\le -\eps'$ for $t\in [1,+\infty)$.  In particular,  $k(t)\rightarrow -\infty$ as $t\rightarrow +\infty$.  On the other hand,  $g_{\eps}\ge (1-\eps)f_{\infty}>0$ on $\Gamma_{\infty}$,  which is a contradiction.

Finally we observe that $\lim_{R\rightarrow +\infty}g_{\eps}(R\lambda)=+\infty$,  which follows from $g_{\eps}\ge (1-\eps)f_{\infty}$ as well as Lemma \ref{l2.7}.
\end{proof}

Now it only remains to establish Proposition \ref{p3.4} which we are going to use Richberg technique.  
The key ingredient in the Richberg technique is the notion of regularized maximum.  Define:
\begin{equation}
M(t)=\max\big(t_1,\cdots,t_m),\,\,\,(t_1,\cdots,t_m)\in \bR^m. 
\end{equation}
Let $\varphi(s)\in C_c^{\infty}(\bR)$ be an even function such that $\varphi\ge 0$,  $supp\,\varphi\subset [-1,1]$ and $\int_{\bR}\varphi(s)ds=1$.  For each $(\eps_1,\cdots,\eps_m)\in \bR_{>0}^m$,  we put 
\begin{equation*}
\varphi_{\eps}(y)=\varphi(\frac{y_1}{\eps_1})\cdots \varphi(\frac{y_m}{\eps_m})\frac{1}{\eps_1\cdots \eps_m},\,\,y\in \bR^m.
\end{equation*}
We then define:
\begin{equation*}
M_{\eps}(t)=\int_{\bR^m}M(t+y)\varphi_{\eps}(y)dy,\,\,\,t\in \bR^m.
\end{equation*}
The function $M_{\eps}(t)$ is called the regularized maximum of $t_1,\,t_2,\cdots,t_m$.
The following lemma is taken from Harvey-Lawson-Pliś \cite{HLP},  Properties 2.1:
\begin{lem}\label{l3.4}
\begin{enumerate}
\item $M_{\eps}(t)$ is monotone increasing in all variables,  smooth and convex on $\bR^m$,
\item $M_{\eps}(t)$ is symmetric in the following sense: for any $i,\,j\in \{1,\cdots,m\}$,  if we swap $\eps_i$ and $\eps_j$,  and $t_i$ and $t_j$,  the value of $M_{\eps}(t)$ does not change.
\item $M_{\eps}(t+se)=M_{\eps}(t)+s$,  where $s\in \bR$ and $e=(1,\cdots,1)\in \bR^m$.  As an immediate consequence:
\begin{equation*}
\sum_{j=1}^m\frac{\partial M_{\eps}}{\partial t_j}\equiv 1.
\end{equation*}
\item $M(t)\le M_{\eps}(t)\le M(t+\eps)$.  Here $\eps=(\eps_1,\cdots,\eps_m)$.
\item If $t_j+\eps_j\le \max_{i\neq j}(t_i-\eps_i)$,  then
\begin{equation*}
M_{\eps}(t)=M_{\eps_1,\cdots,\hat{\eps}_j,\cdots,\eps_m}\big(t_1,\cdots,\hat{t}_j,\cdots,t_m\big).
\end{equation*}
Here $\hat{\eps}_j$ and $\hat{t}_j$ means this component is skipped.
\end{enumerate}
\end{lem}
Using the regularized maximum,  we can formulate the following gluing lemma:
\begin{lem}\label{l3.6New}
Let $\Omega$,  $u$,  $h$ be as in Proposition \ref{p3.4}.  Assume that there exists a locally finite family of bounded open sets $\Omega_{\alpha}$ with $\bar{\Omega}_{\alpha}\subset \Omega$ together with the following properties hold:
\begin{enumerate}
\item Each $\Omega_{\alpha}$ is convex and contains a compact subset $K_{\alpha}$ such that $\Omega\subset \cup_{\alpha}K_{\alpha}$.  Moreover,  $\sup_{\bar{\Omega}_{\alpha}}u\le \inf_{\bar{\Omega}_{\alpha}}(u+\frac{h}{2})$.
\item For each $\alpha$,  there exists $v_{\alpha}\in C^{\infty}(\bar{\Omega}_{\alpha})$ and $0<\eps_{\alpha}<\frac{1}{2}\inf_{\bar{\Omega}_{\alpha}}h$,  such that $v_{\alpha}-\eps_{\alpha}>u$ on $K_{\alpha}$,  $v_{\alpha}+\eps_{\alpha}<u$ on $\partial \Omega_{\alpha}$.  Moreover,  each $v_{\alpha}$ satisfies $D^2v_{\alpha}\ge \eps_{\alpha}I$ on $\Omega_{\alpha}$.
\end{enumerate}
We define:
\begin{equation*}
\tilde{u}(x)=M_{\eps}\{v_{\alpha}(x):x\in \Omega_{\alpha}\},\,\,\eps=(\eps_{\alpha})_{\alpha}.
\end{equation*}

Then one has:
\begin{enumerate}
\item $\tilde{u}\in C^{\infty}(\Omega)$,  $D^2\tilde{u}>0$ on $\Omega$,
\item $u\le \tilde{u}\le u+h$.
\end{enumerate}
\end{lem}
\begin{proof}
First we prove that $\tilde{u}\in C^{\infty}(M)$.  It will be sufficient to show that,  for any $x\in M$,  there is a finite subset $\mathcal{I}_x$ of the index set $\alpha$,  and a neighborhood $U$ of $x$,  such that:
\begin{equation*}
\tilde{u}(y)=M_{\eps}\{v_{\alpha}(y):\alpha\in \mathcal{I}_x\},\,\,\,y\in U.
\end{equation*}
Indeed,  we are going to take $\mathcal{I}_x$ to be the set of $\alpha$ such that $x\in \Omega_{\alpha}$.  Because of the local finiteness assumption,  we see that $\mathcal{I}_x$ must be finite.
Likewise,  we define $\mathcal{I}'_x$ to be the set of $\alpha$ such that $x\in \partial \Omega_{\alpha}$,  and $\mathcal{I}''_x$ to be the set of $\alpha$ such that $x\in \bar{\Omega}_{\alpha}^c$.  Using the local finiteness of $\Omega_{\alpha}$,  we can find a neighborhood $U_1$ of $x$ such that
\begin{enumerate}
\item $\bar{U}_1$ only intersects with finitely many $\Omega_{\alpha}$,
\item $U_1\subset \cap_{\alpha\in \mathcal{I}_x}\Omega_{\alpha}$,  $U_1\subset \cap_{\alpha\in \mathcal{I}''_x}\bar{\Omega}_{\alpha}^c$.
\end{enumerate}

 That is, for $y\in U_1$,
\begin{equation*}
w(y)=M_{\eps}\{v_{\alpha}(y):\alpha\in \mathcal{I}_x\cup \mathcal{I}_x',\,\,y\in \Omega_{\alpha}\}.
\end{equation*}

For any $\alpha\in \mathcal{I}_x'$,  one has $x\in \partial \Omega_{\alpha}$,  so that by assumption (2),  one has:
\begin{equation*}
v_{\alpha}(x)+\eps_{\alpha}<u(x).
\end{equation*}
On the other hand,  since $M\subset\cup_{\alpha}K_{\alpha}$,  there exists $\beta$ such that $x\in K_{\beta}$,  which implies:
\begin{equation*}
v_{\beta}(x)-\eps_{\beta}>u(x).
\end{equation*}
Therefore,  one gets:
\begin{equation*}
v_{\alpha}(x)+\eps_{\alpha}<\max_{\alpha'\in \mathcal{I}_x}(v_{\alpha'}(x)-\eps_{\alpha'}).
\end{equation*}
Hence by continuity of $v_{\alpha}$,  there is a neighborhood $U_2$ of $x$,  contained in $U_1$,  such that for any $\alpha\in \mathcal{I}_x'$ (note that $\mathcal{I}_x'$ is a finite set because of local finiteness of the covering) and any $y\in U_2\cap \bar{\Omega}_{\alpha}$,  one has
\begin{equation}
v_{\alpha}(y)+\eps_{\alpha}\le \max_{\alpha'\in \mathcal{I}_x}(v_{\alpha'}(y)-\eps_{\alpha'}),\,\,\,\alpha\in \mathcal{I}_x'.
\end{equation}
Hence we may use item (5) of Lemma \ref{l3.4} to conclude that for $y\in U_2$,  $\alpha\in \mathcal{I}_x'$,  $v_{\alpha}(y)$ can be skipped when taking the regularized maximum,  that is:
\begin{equation*}
\tilde{u}(y)=M_{\eps}\{v_{\alpha}(y):\alpha\in \mathcal{I}_x\},\,\,\,y\in U_2.
\end{equation*}
Now we are done proving (1).   Moreover,  we know from local finiteness that for any $x_0\in \Omega$,  there is a neighborhood $U$ of $x_0$,  such that:
\begin{equation*}
\tilde{u}(y)=M_{\eps}\{v_{\alpha_j}(y):1\le j\le m\},\,\,\,\,y\in U.
\end{equation*}
Then we can compute the Hessian of $\tilde{u}$:
\begin{equation*}
\tilde{u}_{x_ix_j}(x_0)=\sum_{p=1}^m\frac{\partial M_{\eps}}{\partial t_p}(v_{\alpha_p})_{x_ix_j}+\sum_{p,q=1}^m\frac{\partial^2M_{\eps}}{\partial t_p\partial t_q}(v_{\alpha_p})_{x_i}(v_{\alpha_q})_{x_j}.
\end{equation*}
Using that $\frac{\partial M_{\eps}}{\partial t_p}\ge 0,\,\,\sum_{p=1}^m\frac{\partial M_{\eps}}{\partial t_p}=1$,  and that $M_{\eps}$ is convex from Lemma \ref{l3.4},   
we can conclude that $D^2\tilde{u}(x_0)\ge \sum_{p=1}^m\frac{\partial M_{\eps}}{\partial t_p}D^2v_{\alpha_p}\ge \sum_{p=1}^m\frac{\partial M_{\eps}}{\partial t_p}\eps_pI>0.$

Next we show that $u\le \tilde{u}\le u+h$.  For any $x_0\in \Omega$,  there exists $\beta$ so that $x_0\in K_{\beta}$ and we have $u(x_0)<v_{\beta}(x_0)-\eps_{\beta}$.  In particular:
\begin{equation*}
u(x_0)\le M\{v_{\alpha}(x_0):x_0\in \Omega_{\alpha}\}\le M_{\eps}\{v_{\alpha}(x_0):x_0\in \Omega_{\alpha}\}=\tilde{u}(x_0).
\end{equation*}
To prove the other inequality,  we have:
\begin{equation*}
\begin{split}
&\tilde{u}(x_0)=M_{\eps}\{v_{\alpha_j}(x_0):1\le j\le m\}\le M\{v_{\alpha_j}(x_0)+\eps_{\alpha_j}:1\le j\le m\}\\
&\le M\{\sup_{\bar{\Omega}_{\alpha_j}}(v_{\alpha_j}(x)+\eps_{\alpha_j}):1\le j\le m\}=M\{\sup_{\partial \Omega_{\alpha_j}}(v_{\alpha_j}(x)+\eps_{\alpha_j}):1\le j\le m\}\\
&\le M\{\sup_{\partial \Omega_{\alpha_j}}(u(x)+\eps_{\alpha_j}):1\le j\le m\}\le M\{\inf_{\bar{\Omega}_{\alpha_j}}(u(x)+\frac{h(x)}{2})+\eps_{\alpha_j}\}\le u(x_0)+h(x_0).
\end{split}
\end{equation*}
In the equality above on the second line,  we used that each $v_{\alpha_j}$ is convex,  so $\sup$ must be achieved on the boundary.  In the first inequality of the last line,  we used that $v_{\alpha_j}(x)<u(x)$ on $\partial \Omega_{\alpha_j}$.  In the last inequality of the last line,  we used that each $\eps_{\alpha_j}\le \frac{1}{2}\inf_{\bar{\Omega}_{\alpha_j}}h\le \frac{1}{2}h(x_0)$.
\end{proof}

Our next goal is to construct the family of open sets $\Omega_{\alpha}$.  For this we observe that:
\begin{lem}\label{l3.7New}
Let $u$,  $\Omega$ be as in Proposition \ref{p3.4}.  Let $x_0\in \Omega$.  Let $l_{x_0}$ be an affine function such that $\{x\in \Omega:l_{x_0}(x)=u(x)\}=\{x_0\}$.  For any $t>0$,  we define:
\begin{equation}\label{3.8NNew}
S_t(x_0):=\{x\in \Omega:u(x)<l_{x_0}(x)+t\}.
\end{equation}
Then one has $\diam\,S_t(x_0)\rightarrow 0$ as $t\rightarrow 0+$.
\end{lem}
\begin{proof}
If not,  there there is a sequence $t_k\rightarrow 0+$,  and $\eta>0$ such that $\diam\,S_{t_k}(x_0)\ge \eta$ for all $k$.  Hence we may choose $x_k,\,y_k\in S_{t_k}(x_0)$,  $|x_k-y_k|\ge \frac{\eta}{2}$.  That is,  $u(x_k)<l_{x_0}(x_k)+t_k,\,\,u(y_k)<l_{x_0}(y_k)+t_k$.  From convexity,  the segment $[x_k,y_k]\subset S_{t_k}(x_0)$.  Passing to the limit,  we see that $\{x\in \Omega:l_{x_0}(x)=u(x)\}$ will contain a nontrivial segment with length $\ge \frac{\delta}{2}$,  which is a contradiction.
\end{proof}
Now we are ready to prove Proposition \ref{p3.4}:
\begin{proof}
For each $x_0\in \Omega$,  we consider $S_t(x_0)$ defined by (\ref{3.8NNew}).  From Lemma \ref{l3.7New},  we know that $S_t(x_0)$ will be compactly contained in $\Omega$ if $t$ is small enough.  Moreover,  from Lemma \ref{l3.7New},  we see that with $t$ small enough one has $\sup_{\bar{S}_t(x_0)}u\le \inf_{\bar{S}_t(x_0)}(u+\frac{h}{2})$.  This is possible because of the continuity of $u$ and $h$ on $\Omega$ and that $h$ is strictly positive on $\Omega$.  Denote this small enough $t$ to be $t_{x_0}$.  

Next we wish to find a sub-family of $\{S_{t_{x_0}}(x_0)\}_{x_0\in \Omega}$ which is locally finite on $\Omega$,  after possibly decreasing $t_{x_0}$ further.  For this,  we can choose a sequence of open subsets $\Omega_j\subset \Omega,\,j\ge 1$ with the following properties:
\begin{enumerate}
\item For each $j\ge 1$,  $\Omega_j$ is bounded,  with $\bar{\Omega}_j\subset \Omega_{j+1}$,  and $\cup_j\Omega_j=\Omega$,
\item For $j\ge 3$,  $\bar{\Omega}_j-\Omega_{j-1}$ is a compact subset of $\Omega_{j+1}-\bar{\Omega}_{j-2}$.
\end{enumerate}

First we consider $\bar{\Omega}_2$,  which is compact,  and we have a covering $\bar{\Omega}_2\subset \cup_{x_0\in \bar{\Omega}_2}S_{\frac{t_{x_0}}{2}}(x_0)$.  So we may take a finite sub-family $S_{\frac{t_j}{2}}(x_j)$ that covers $\bar{\Omega}_2$.  Next for each $l\ge3$,  we consider $\bar{\Omega}_l-\Omega_{l-1}$.  For $x_0\in \bar{\Omega}_l-\Omega_{l-1}$,  we may need to decrease $t_{x_0}$ so as to make sure that:
\begin{equation}\label{3.10NNew}
\bar{S}_{t_{x_0}}(x_0)\subset \Omega_{l+1}-\bar{\Omega}_{l-2}.
\end{equation}
This is possible because of Lemma \ref{l3.7New}.
Then we take a finite family $S_{\frac{t_j}{2}}(x_j)$ with $x_j\in \bar{\Omega}_l-\Omega_{l-1}$ which covers $\bar{\Omega}_l-\Omega_{l-1}$.    Moreover,  $\bar{S}_{t_j}(x_j)\subset \Omega_{l+1}-\bar{\Omega}_{l-2}$.
We combine all the $S_{t_j}(x_j)$ obtained this way,  and it is a locally finite family of open sets,  due to (\ref{3.10NNew}).

Now we have a covering of $\Omega$ by $S_{\frac{t_j}{2}}(x_j),\,j\ge 1$.  We now wish to take $\{\Omega_{\alpha}\}$ to be $S_{t_j}(x_j)$ and $K_{\alpha}=\bar{S}_{\frac{t_j}{2}}(x_j)=\{x\in \Omega:u(x)\le l_{x_j}(x)+\frac{t_j}{2}\}$.  Then we wish to take 
\begin{equation*}
v_j(x)=l_{x_j}(x)+\frac{3t_j}{4}+\eps_j|x-x_j|^2.
\end{equation*}
We see that,  with $\eps_j>0$ chosen small enough,  we can guarantee that $v_j(x)>u(x)+\eps_j$ on $K_j$,  $v_j(x)<u(x)-\eps_j$ on $\partial \Omega_j$ and $\eps_j<\frac{1}{2}\inf_{\bar{\Omega}_j}h$.  Moreover,  $D^2v_j\ge \eps_jI$.
Therefore,  we have verified the assumptions of Lemma \ref{l3.6New} and it gives us the function asserted in Proposition \ref{p3.4}.
\end{proof}

\subsection{Regularization of the strict viscosity subsolution}
The goal of this section is to prove Proposition \ref{p3.2}.

For the remainder of this section,  we will simply denote $g_{\eps_1}$ by $g$ for simplicity.  We have the following gluing lemma which is the analogue of Lemma \ref{l3.6New}.

\begin{lem}\label{l3.5}
Let $\eps_{1.1}>0$.  Let $u_1$ be as assumed in Proposition \ref{p3.2}.  Assume that there exists a finite family of open sets $\Omega_{\alpha}\subset M$ with the following properties hold:
\begin{enumerate}
\item Each $\Omega_{\alpha}$ has smooth boundary,  and contains a compact subset $K_{\alpha}$ such that $M\subset \cup_{\alpha}K_{\alpha}$.  Moreover,  $\sup_{\bar{\Omega}_{\alpha}}u_1\le \inf_{\bar{\Omega}_{\alpha}}u_1+\eps_{1.1}.$
\item For each $\alpha$,  there exists $v_{\alpha}\in C^{\infty}(\bar{\Omega}_{\alpha})$ and $\eps_{\alpha}\in (0,\eps_{1.1}]$,  such that $v_{\alpha}-\eps_{\alpha}>u_1$ on $K_{\alpha}$,  $v_{\alpha}+\eps_{\alpha}<u_1$ on $\partial \Omega_{\alpha}$.  Moreover,  there exists $\delta>0$,  such that for all $\alpha$:
\begin{equation*}
\lambda[\chi+dd^cv_{\alpha}]\in \Gamma_{\infty},\,\,\,g\big(\lambda[\chi+dd^cv_{\alpha}]\big)\ge e^{G+c}+\delta.
\end{equation*}
\item For each $\alpha$,  $\sup_{\bar{\Omega}_{\alpha}}v_{\alpha}\le \sup_{\partial \Omega_{\alpha}}v_{\alpha}+\eps_{1.1}$.
\end{enumerate}
We define:
\begin{equation*}
w(x)=M_{\eps}\{v_{\alpha}(x):x\in \Omega_{\alpha}\},\,\,\eps=(\eps_{\alpha})_{\alpha}.
\end{equation*}
Then one has:
\begin{enumerate}
\item $w\in C^{\infty}(M)$.
\item $\lambda[\chi+dd^cw]\in \Gamma_{\infty}$ and $g\big(\lambda[\chi+dd^cw]\big)\ge e^{G+c}+\delta$.
\item $u_1\le w\le u_1+3\eps_{1.1}$.
\end{enumerate}
\end{lem}
The point of this lemma reduces the construction of a global smooth subsolution to the construction of $v_{\alpha}$'s,  which is purely local.  The proof essentially follows the Theorem 1.2 of  \cite{HLP} carried out in our special setting.
\begin{proof}
(of Lemma \ref{l3.5})
First we prove that $w\in C^{\infty}(M)$.  The proof follows word-for-word as the proof of the item (1) in Lemma \ref{l3.6New}.  In this proof,  we will only need to use $M\subset \cup_{\alpha}K_{\alpha}$ and $v_{\alpha}-\eps_{\alpha}>u_1$ on $K_{\alpha}$,  $v_{\alpha}+\eps_{\alpha}<u_1$.

Next we prove (2).  Let $x_0\in M$,  we have shown that there is a neighborhood $U$ of $x_0$ and a subfamily of $\Omega_{\alpha}$,  denoted as $\Omega_{\alpha_j},\,1\le j\le m$,  such that:
\begin{equation*}
w(y)=M_{\eps}\{v_{\alpha_j}(y):1\le j\le m\}.
\end{equation*}
Then we may calculate:
\begin{equation*}
dd^cw=\sum_{p,q=1}^m\frac{\partial^2M_{\eps}}{\partial t_p\partial t_q}dv_{\alpha_p}\wedge d^cv_{\alpha_q}+\sum_{p=1}^m\frac{\partial M_{\eps}}{\partial t_p}dd^cv_{\alpha_p}.
\end{equation*}
Therefore,
\begin{equation*}
\chi+dd^cw=\sum_{p,q=1}^m\frac{\partial^2M_{\eps}}{\partial t_p\partial t_q}dv_{\alpha_p}\wedge d^cv_{\alpha_q}+\sum_{p=1}^m\frac{\partial M_{\eps}}{\partial t_p}(\chi+dd^cv_{\alpha_p})\ge \sum_{p=1}^m\frac{\partial M_{\eps}}{\partial t_p}(\chi+dd^cv_{\alpha_p}).
\end{equation*}
Here we used the convexity of $M_{\eps}$.  Since one has $\frac{\partial M_{\eps}}{\partial t_p}\ge 0,\,\,1\le p\le m$,  and $\sum_{p=1}^m\frac{\partial M_{\eps}}{\partial t_p}=1$,  we see that $\lambda[\chi+dd^cw]\in \Gamma_{\infty}$,  using that $\lambda[\chi+dd^cv_{\alpha_p}]\in \Gamma_{\infty}$ for each $p$.  Here we are also using the elementary fact that $\{A\text{ is $n\times n$ Hermitian matrix}:\lambda(A)\in \Gamma_{\infty}\}$ is also a closed convex cone (in the space of Hermitian matrices).

Next,  we have:
\begin{equation*}
\begin{split}
&g\big(\lambda[\chi+dd^cw]\big)\ge g\big(\lambda[\sum_{p=1}^m\frac{\partial M_{\eps}}{\partial t_p}(\chi+dd^cv_{\alpha_p})]\big)\ge \sum_{p=1}^m\frac{\partial M_{\eps}}{\partial t_p}g\big(\lambda[\chi+dd^cv_{\alpha_p}]\big)\\
&\ge e^{G+c}+\delta.
\end{split}
\end{equation*}
The first inequality is due to that since $\chi+dd^cw\ge \sum_{p=1}^m\frac{\partial M_{\eps}}{\partial t_p}(\chi+dd^cv_{\alpha_p})$,  one has 
\begin{equation*}
\lambda(\chi+dd^cw)=\lambda\big(\sum_{p=1}^m\frac{\partial M_{\eps}}{\partial t_p}(\chi+dd^cv_{\alpha_p})\big)+\tilde{\lambda},\,\,\,\tilde{\lambda}\in \Gamma_n.
\end{equation*}
The second inequality used that since $g(\lambda)$ is concave,  $A\mapsto g(\lambda(A))$ is also concave,  as a function in the space of Hermitian matrices.  Now we have proved (2).  

Finally we verify (3).  We fix $x_0\in M$.
To see $u\le w$,  we note that there exists $\beta$ so that $x_0\in K_{\beta}$ and we have $u(x_0)<v(x_0)-\eps_{\beta}$.  In particular:
\begin{equation*}
u(x_0)\le M\{v_{\alpha}(x_0):x_0\in \Omega_{\alpha}\}\le M_{\eps}\{v_{\alpha}(x_0):x_0\in \Omega_{\alpha}\}=w(x_0).
\end{equation*}
The second inequality used item (4) in Lemma \ref{l3.4}.  To prove the other inequality,  first we can write $w(x_0)=M_{\eps}\{v_{\alpha_j}(x_0):1\le j\le m\}$.  We have
\begin{equation*}
\begin{split}
&w(x_0)=M_{\eps}\{v_{\alpha_j}(x_0):1\le j\le m\}\le M\{v_{\alpha_j}(x_0)+\eps_{\alpha_j}:1\le j\le m\}\\
&\le M\{\sup_{x\in \partial \Omega_{\alpha_j}}v_{\alpha_j}(x)+\eps_{1.1}+\eps_{\alpha_j}:1\le j\le m\}\le M\{\sup_{x\in \partial \Omega_{\alpha_j}}u(x)+\eps_{1.1}+\eps_{\alpha_j}:1\le j\le m\}\\
&\le M\{\inf_{\Omega_{\alpha_j}}u(x)+2\eps_{1.1}+\eps_{\alpha_j}:1\le j\le m\}\le u(x_0)+3\eps_{1.1}.
\end{split}
\end{equation*}
In the first inequality above,  we used item (4) of Lemma \ref{l3.4}.  In the second inequality,  we used item (3) of the assumption that $\sup_{\bar{\Omega}_{\alpha}}v_{\alpha}\le \sup_{\partial \Omega_{\alpha}}v_{\alpha}+\eps_{1.1}$.  In the third inequality,  we used that $v_{\alpha}<u$ on $\partial \Omega_{\alpha}$.  In the fourth inequality,  we used item (1) of the assumption that $\sup_{\bar{\Omega}_{\alpha}}u_1\le \inf_{\bar{\Omega}_{\alpha}}u_1+\eps_{1.1}$.  The last inequality used that $x_0\in \Omega_{\alpha_j}$ and that $\eps_{\alpha_j}\le \eps_{1.1}$.
\end{proof}
Our next goal is to construct the local strict subsolution $v_{\alpha}$.  We are going to do it via solving a Dirichlet problem of complex Hessian equation defined by the operator $g$.
More precisely,  we are going to need the following lemma:
\begin{lem}\label{l3.6N}
Let $x_0\in M$,  we choose normal coordinate at $x_0$ such that $(\omega_0)_{i\bar{j}}(x_0)=\delta_{ij}$.  Denote $B_r$ to be the (open) ball centered at $x_0$ with radius $r$ under this coordinate.  Then there exists $r_0>0$ small enough depending only on the background metric,  such that for any $0<r\le r_0$,  and any $\varphi\in C^{\infty}(\partial B_r)$,  $\psi \in C^{\infty}(\bar{B}_r)$ with 
\begin{equation}\label{3.8NNNew}
\inf_{\bar{B}_r}\psi>\lim\sup_{\lambda\rightarrow \partial \Gamma_{\infty}}g,
\end{equation}

there is a unique  solution to the Dirichlet problem:
\begin{equation*}
g\big(\lambda[\chi+dd^cv]\big)=\psi\text{ in $B_r$}, \,\,\,v=\varphi\text{ on $\partial B_r$},\,\,\lambda[\chi+dd^cv]\in \Gamma_{\infty},\,\,v\in C^{\infty}(\bar{B}_r).
\end{equation*}
\end{lem}
This result follows from a result by Yuan \cite{Y},  Theorem 9.2:
\begin{thm}\label{t3.10}
Let $M$ be a Hermitian manifold with boundary,  equipped with the Hermitian metric $\omega_0$.  Let $\psi\in C^{\infty}(M)$,  $\varphi\in C^{\infty}(\partial M)$.  We consider the following Dirichlet problem:
\begin{equation}\label{3.8New}
f\big(\lambda[\chi+dd^cu]\big)=\psi,\,\,\,\lambda[\chi+dd^cu]\in \Gamma\,\,\text{ in $M$},\,\,\,\,u=\varphi \text{ on $\partial M$}.
\end{equation}
Assume that 
\begin{enumerate}
\item $f$ and $\Gamma$ satisfies Assumption \ref{a1.1N} and \ref{chi positivity-2}.  Assume also that $\inf_M\psi>\lim\sup_{\lambda\rightarrow \partial\Gamma}f(\lambda)$.
\item There exists a subsolution $u_0\in C^{\infty}(M)$ in the following sense:
\begin{equation*}
f\big(\lambda[\chi+dd^cu_0]\big)\ge \psi,\,\,\lambda[\chi+dd^cu_0]\in \Gamma\text{ in $M$},\,\,\,u_0=\varphi\text{ on $\partial M$}.
\end{equation*}
\item For any $x_0\in \partial M$,  there exists $R>0$ sufficiently large,  such that $(-\kappa_1,\cdots,-\kappa_{n-1},R)\in \Gamma_{\Mp}^f$,  where $\kappa_1,\cdots,\kappa_{n-1}$ are eigenvalues of the Levi form $L_{\partial M}$ with respect to $\omega_0|_{\partial M}$,  and $\Gamma_{\Mp}^f=\{\lambda\in \Gamma:\lim_{t\rightarrow \infty}f(t\lambda)>-\infty\}$.
\end{enumerate}
Then (\ref{3.8New}) has a unique solution which is smooth up to the boundary.
\end{thm}
Next let us check that one may apply Theorem \ref{t3.10} in our case:
\begin{proof}
(of Lemma \ref{l3.6N}) In Theorem \ref{t3.10},  we are going to take the manifold with boundary to be $\bar{B}_r$,  which is the coordinate ball centered at $x_0$.  Also the $\Gamma$ in Theorem \ref{t3.10} is going to be taken to be $\Gamma_{\infty}$,  and $f$ will be $g$.  The assumption (\ref{3.8NNNew}) guarantees the non-degeneracy condition.  To verify (2),  we just need to extend $\varphi$ smoothly on $\bar{B}_r(0)$,  then we take $u_0=A(|x|^2-r^2)+\varphi$.  It is clear that $u_0=\varphi$ on $\partial B_r(0)$.   With $A$ chosen sufficiently large,  one would have $\lambda[\chi+dd^cu_0]\in \Gamma$ and $f\big(\lambda[\chi+dd^cu_0]\big)\rightarrow +\infty$ as $A_0\rightarrow \infty$.  So we have (2).  To see (3),  we note that the Levi form of $\partial B_r$ will be negative definite,  and $\Gamma_{\mathcal{G}}^f$ will be $\Gamma_{\infty}$ in our case,  because $\lim_{R\rightarrow+\infty}g(R\lambda)=+\infty$ by Proposition \ref{p3.2N}.
\end{proof}

Using this solvability result,  we can construct the functions $v_{\alpha}$ assumed in Lemma \ref{l3.5},  if we take $\Omega_{\alpha}$ to be balls with sufficiently small radius under local coordinates.  More precisely,  we have:
\begin{lem}\label{l3.7}
Let $x_0\in M$,  we choose normal coordinate at $x_0$ such that $(\omega_0)_{i\bar{j}}(x_0)=\delta_{ij}$.  Denote $B_r$ to be the (open) ball centered at $x_0$ with radius $r$ under this coordinate. 
Let $\underline{u}\in C(M)$ be as assumed in Proposition \ref{p3.2} for some $\delta_1>0$.  Then there exists $r_0>0$ small enough depending only on the background metric,  such that for any $0<r'<r\le r_0$,  and any $\eps>0$,  there exists $v_{r',r}\in C^{\infty}(\bar{B}_r(0))$,  such that:
\begin{enumerate}
\item $\lambda[\chi+dd^cv_{r',r}]\in \Gamma_{\infty}$,  $g\big(\lambda[\chi+dd^cv_{r',r}]\big)\ge e^{G(x)+c}+\delta_1-\eps$,
\item $v_{r',r}>\underline{u}$ on $\bar{B}_{r'}$,  $v_{r',r}<\underline{u}$ on $\partial B_r$.
\item There exists a universal constant $C>0$ depending only on the form $\chi$ as well as the background metric,  such that:
\begin{equation*}
\sup_{B_r}v_{r',r}\le \sup_{\partial B_r}v_{r',r}+Cr^2.
\end{equation*}
\end{enumerate}
\end{lem}
\begin{proof}
Without loss of generality,  we can assume that $x_0=0$ under local coordinates.  Denote $\rho(z)=|z|^2-r^2$,  so that $\rho<0$ in $B_r$ and $\rho=0 $ on $\partial B_r$.  

Let $s>0$ be small enough such that $0<s<\inf_{\bar{B}_{r'}}(-\rho)$.  Moreover,  we can take $\eps_5>0$ small enough so that $\eps_5dd^c\rho\le \eps_2'\omega_0$,  where $\eps_2'>0$ is the constant appearing in (\ref{3.1NN}).  With this choice of $\eps_5$,  we put $w=u_1-\eps_5(\rho+s)$,  we see that $w$ satisfies the following in the viscosity sense:
\begin{equation}\label{3.6New}
\lambda[\chi+dd^cw]\in \Gamma_{\infty},\,\,\,g\big(\lambda[\chi+dd^cw]\big)\ge e^{G(x)+c}+\delta_1.
\end{equation}

 Since $u\in C(M)$,  we can find $\varphi_k\in C^{\infty}(\partial B_r)$,  such that $\varphi_k\rightarrow u_1$ uniformly on $\partial B_r$.  Define $v_k$ to be the solution to the following Dirichlet problem:
\begin{equation}\label{3.7NNN}
\begin{split}
&g\big(\lambda[\chi+dd^cv_k]\big)=e^{G(x)+c}+\delta_1-\frac{1}{k}\,\text{ in $B_r$},\,\,\,v_k=\varphi_k-\frac{1}{2}\eps_5s\text{ on $\partial B_r$},\\
&\lambda[\chi+dd^cv_k]\in \Gamma_{\infty},\,\,\,v_k\in C^{\infty}(\bar{B}_r).
\end{split}
\end{equation}
The solution exists by Lemma \ref{l3.6N}.  To see that the nondegeneracy condition (\ref{3.8NNNew}) holds,  we note that here $\psi=e^{G+c}+\delta-\frac{1}{k}$,  and $g\le \frac{\pi}{2}\eps_1$ near $\partial \Gamma_{\infty}$,  so that (\ref{3.8NNNew}) holds if $k$ is sufficiently large and $\eps_1$ sufficiently small.

The plan is to take $v_{r',r}=v_k$,  with $k$ sufficiently large.  Now we verify requirements (1),   (2) and (3).  The requirement (1) is automatic,  by our very definition of the Dirichlet problem.  

To see requirement (2),  we note that $\varphi_k\rightarrow u_1$ uniformly on $\partial B_r$,  we see that $v_k<u_1$ for $k$ large enough on $\partial B_r$.  On the other hand,  we claim that:
\begin{equation}\label{3.7NN}
w\le v_k,\text{ for $k$ large enough}.
\end{equation}  
Once we have this,  we see that on $\bar{B}_{r'}$ (note $\rho+s<0$ on $\bar{B}_{r'}$):
\begin{equation*}
u_1<w=u_1-\eps_5(\rho+s)\le v_k.
\end{equation*}
It only remains to verify (\ref{3.7NN}).  First we observe that on $\partial B_r$,  one has $w=u_1-\eps_5s<\varphi_k-\frac{1}{2}\eps_5s$,  if $k$ is large enough.  If $w-v_k$ has positive maximum in $B_r$,  then $v_k+c_k$ touches $w$ from above at the maximum point,  for some $c_k\in \bR$.  But then from (\ref{3.6New}) (as well as the definition of viscosity subsolution),  we see that the following must hold at the maximum point:
\begin{equation*}
g\big(\lambda[\chi+dd^cv_k]\big)\ge e^{G+c}+\delta_1.
\end{equation*}
This is in direct violation with (\ref{3.7NNN}).  
So one must have $w\le v_k$.  Now we prove (3).  We use that $\Gamma_{\infty}\subset \Gamma_1$.  Since $\lambda[\chi+dd^cv_k]\in \Gamma_{\infty}\subset \Gamma_1$,  we see that:
\begin{equation*}
tr_{\omega_0}(\chi+dd^cv_k)=g^{i\bar{j}}\partial_{i\bar{j}}v_k+tr_{\omega_0}\chi\ge 0.
\end{equation*}
That is,  there exists $C_1>0$,  such that $g^{i\bar{j}}\partial_{i\bar{j}}v_k\ge -C_1$.  Hence we may use Alexandrove maximum principle on $\bar{B}_r$ to conclude that:
\begin{equation*}
\sup_{B_r}v_k\le \sup_{\partial B_r}v_k+C_n\diam\,B_r\bigg(\int_{B_r}\frac{C_1^n}{(\det g_{i\bar{j}})^2}dvol\bigg)^{\frac{1}{2n}}\le \sup_{\partial B_r}v_k+Cr^2.
\end{equation*}
\end{proof}
Now we are in a position to prove Proposition \ref{p3.2}.  
\begin{proof}
(of Proposition \ref{p3.2}) For any $x_0\in M$,  we can choose a normal coordinate centered at $x_0$.  Also we denote $\tilde{r}_{0,x_0}>0$ to be the radius determined by Lemma \label{l3.6} and \ref{l3.7}.  Let $\eps>0$.  
Let $r_{x_0}\le \tilde{r}_{0,x_0}$ be chosen sufficiently small,  such that:
\begin{enumerate}
\item $\sup_{\bar{B}_{r_{x_0}}}u_1\le \inf_{\bar{B}_{r_{x_0}}}u_1+\eps$.  This is possible because of the continuity of $u_1$.
\item $\sup_{\bar{B}_{r_{x_0}}}v_{\frac{1}{2}r_{x_0},r_{x_0}}\le \sup_{\partial B_{r_{x_0}}}v_{\frac{1}{2}r_{x_0},r_{x_0}}+\eps$.  Here $v_{\frac{1}{2}r_{x_0},r_{x_0}}$ is the $v_{r',r}$ constructed in Lemma \ref{l3.7} with $r'=\frac{r_{x_0}}{2}$,  $r=r_{x_0}$ and $g\big(\lambda[\chi+dd^cv_{\frac{r_{x_0}}{2},r_{x_0}}]\big)\ge e^{G+c}+\delta_1-\eps$.  This is possible because of item (3) of Lemma \ref{l3.7}.
\end{enumerate}
Now we have $M\subset \cup_{x_0\in M}B_{\frac{r_{x_0}}{2}}$.  Hence there is a finite sub-family,  which we denote as $B_{\frac{r_j}{2}}(x_j)$ such that $M\subset \cup_{j=1}^{K}B_{\frac{r_j}{2}}(x_j)$.

We hope to use Lemma \ref{l3.5} to get our desired global subsolution.  First we take $\eps_{1.1}=\eps$.  
Then we can take $\{\Omega_{\alpha}\}_{\alpha}$ to be $B_{r_j}(x_j),\,\,1\le j\le m$,  and $K_{\alpha}$ to be $\bar{B}_{\frac{r_j}{2}}(x_j)$.  We also denote $v_j=v_{\frac{r_j}{2},r_j}$ constructed in Lemma \ref{l3.7}.  From Lemma \ref{l3.7} item (2),  we know that $v_j>u_1$ on $\bar{B}_{\frac{1}{2}r_j}$ and $v_j<u_1$ on $\partial B_{r_j}$.  Therefore,  there exists $\eps_j>0$ such that:
\begin{equation*}
v_j>u_1+\eps_j,\,\,\text{ on $B_{\frac{r_j}{2}}$},\,\,\,\,v_j+\eps_j<u_1,\text{ on $\partial B_{r_j}$}.
\end{equation*}
Moreover,  for each $v_j$,  we have $g\big(\lambda[\chi+dd^cv_j]\big)\ge e^{G+c}+\delta_1-\eps$.

From the smallness of $r_j$,  we see that item (1) of the assumptions of Lemma \ref{l3.5} holds.  We also see that item (2) holds as well,  with $\eps_{\alpha}$'s taken to be $\eps_j$ and $\delta=\delta_1-\eps$.
We also see that item (3) of the assumptions of Lemma \ref{l3.5} also holds.  Hence Lemma \ref{l3.5} will give us a function $w\in C^{\infty}(M)$,  $\lambda[\chi+dd^cw]\in \Gamma_{\infty}$,  $g\big(\lambda[\chi+dd^cw]\big)\ge e^{G+c}+\delta_1-\eps$,  and $u_1\le w\le u_1+3\eps$.
\end{proof}

\section{Stability estimate and existence}
\subsection{Existence}
In this subsection,  we are going to use an approximation argument to obtain a viscosity solution to:
\begin{equation*}
f\big(\lambda[\chi+dd^c\varphi]\big)=e^{G(x)+c},\,\,\,\lambda[\chi+dd^c\varphi]\in \Gamma,
\end{equation*}
for some $c\in \bR$.  Here $G\in C(M)$.  

First we can find a sequence $G_j\in C^{\infty}(M)$ such that $G_j\rightarrow G$ uniformly.  Let $\eps_j\rightarrow 0+$.  First we want to make sure that the equation is solvable with right hand side $G_j$:
\begin{lem}\label{l4.1}
For large enough $j$,  the equation $f\big(\lambda[\chi+dd^c\varphi_j]\big)=e^{G_j(x)+c_j}$ has a smooth solution for some $c_j\in \bR$.
\end{lem}
To prove this,  we need to make sure that the solvability constant $c_j$ will converge to the correct constant:
\begin{lem}\label{l4.2NNew}
Define
\begin{equation*}
e^{c_j}=\inf_{u\in \mathcal{E}_{\Gamma,\chi}}\max_Me^{-G_j(x)}f\big(\lambda[(\chi+dd^cu]\big).
\end{equation*}
and
\begin{equation*}
e^c=\inf_{u\in \mathcal{E}_{\Gamma,\chi}}\max_Me^{-G(x)}f\big(\lambda[\chi+dd^cu]\big).
\end{equation*}
Then one has:
\begin{equation*}
e^c\cdot e^{-||G_j-G||_{L^{\infty}}}\le e^{c_j}\le e^c\cdot e^{||G_j-G||_{L^{\infty}}}.
\end{equation*}
In particular,  one has $c_j\rightarrow c$.
Here we recall that $\mathcal{E}_{\Gamma,\chi}$ denotes the set of smooth functions $u$ on $M$ such that $\lambda[\chi+dd^cu]\in \Gamma$.
\end{lem}
Now we are ready to prove Lemma \ref{l4.1}.
\begin{proof}
(Of Lemma \ref{l4.1})
This essentially follows from the existence of a smooth subsolution constructed in Section 3.  Indeed,  we have shown that for some $\tilde{\delta}_0>0$,  there exists a  smooth function $\tilde{u}$ such that:
\begin{equation}\label{4.1NNew}
f_{\infty}\big(\lambda[\chi+dd^c\tilde{u}]\big)\ge e^{G(x)+c}+\tilde{\delta}_0,\,\,\,\lambda[\chi+dd^c\tilde{u}]\in \Gamma_{\infty},
\end{equation}
where the constant $c$ is given by:
\begin{equation}\label{c}
e^c=\inf_{u\in \mathcal{E}_{\Gamma,\chi}}\max_Me^{-G(x)}f\big(\lambda[\chi+dd^cu]\big).
\end{equation}
The corresponding constant for $G_j$ is given as:
\begin{equation*}
e^{c_j}=\inf_{u\in \mathcal{E}_{\Gamma,\chi}}\max_Me^{-G_j(x)}f\big(\lambda[\chi+dd^cu]\big).
\end{equation*}

We have shown that  $c_j\rightarrow c$.  Since $G_j\rightarrow G$ uniformly,  we see that with $j$ sufficiently large,  one has:
\begin{equation*}
f_{\infty}\big(\lambda[\chi+dd^c\tilde{u}]\big)\ge e^{G_j+c_j}+\frac{1}{2}\tilde{\delta}_0.
\end{equation*}
From Guo-Song \cite{GS},  we know that for such $j$ there exists a smooth solution to:
\begin{equation}\label{4.1NNNew}
f\big(\lambda[\chi+dd^c\varphi_j]\big)=e^{G_j(x)+c_j},\,\,\,\lambda[\chi+dd^c\varphi_j]\in \Gamma,\,\,\,\sup_M\varphi_j=0.
\end{equation}
\end{proof}

The next step is to show that there is a subsequence of $\varphi_j$ that converges uniformly,  thereby we would get a viscosity solution (with right hand side $G(x)$ and constant $c$) in the limit.

First we observe that the sequence $\varphi_j$ is uniformly bounded due to the work by Szekelyhidi \cite{S}.  Using this,  we observe that $\nabla\varphi_j$ is uniformly bounded in $L^2$:
\begin{lem}
There exists a constant $C>0$,  depending only on the uniform $C^0$ bound of $G_j$ as well as the background metric $\omega_0$,  such that:
\begin{equation*}
\int_Md^c\varphi_j\wedge d\varphi_j\wedge\omega_0^{n-1}\le C.
\end{equation*}
\end{lem}
\begin{proof}
We use that each $\varphi_j$ is $\Gamma$-subharmonic,  and that $\Gamma\subset \Gamma_1=\{\lambda\in \bR^n:\sum_i\lambda_i> 0\}.$  This in particular implies:
\begin{equation*}
tr_{\omega_0}(\chi+dd^c\varphi_j)\ge 0.
\end{equation*}
This would imply,  for some $C_1>0$:
\begin{equation*}
tr_{\omega_0}(dd^c\varphi_j)\ge -C_1,
\end{equation*}
which is equivalent to:
\begin{equation}\label{4.1New}
dd^c\varphi_j\wedge \omega_0^{n-1}\ge -\frac{C_1}{n}\omega_0^n.
\end{equation}
Let $C_2>0$ be sufficiently large such that $\varphi_j+C_2>0$ for all $j$,  then we multiply this to both sides of (\ref{4.1New}) to obtain:
\begin{equation*}
-\int_Md\varphi_j\wedge d^c\varphi_j\wedge \omega_0^{n-1}+\int_M(\varphi_j+C_2)d^c\varphi_j\wedge d\omega_0^{n-1}\ge \int_M(-\frac{C_1}{n})(\varphi_j+C_2)\omega_0^n.
\end{equation*}
For the middle term,  we have:
\begin{equation*}
\begin{split}
&\int_M(\varphi_j+C_2)d^c\varphi_j\wedge d\omega_0^{n-1}=\int_Md^c\big(\frac{(\varphi_j+C_2)^2}{2}\big)\wedge d\omega_0^{n-1}=-\int_Md\big(\frac{(\varphi_j+C_2)^2}{2}\big)\wedge d^c\omega_0^{n-1}\\
&=\int_M\frac{(\varphi_j+C_2)^2}{2}dd^c\omega_0^{n-1}.
\end{split}
\end{equation*}
Therefore,  we see that $\int_Md\varphi_j\wedge d^c\varphi_j\wedge \omega_0^{n-1}$ is uniformly bounded from above.
\end{proof}
Therefore we are in a position to apply Rellich compact embedding theorem to extract a subsequence of $\{\varphi_j\}$,  denoted as $\{\varphi_{j_k}\}$,  such that $\varphi_{j_k}$ converges in $L^1$.  In order to get a viscosity solution,  we need to upgrade the $L^1$ convergence to uniform convergence,  which is why we need the stability estimates stated in Theorem \ref{t4.1} below.  Now we explain how to use Theorem \ref{t4.1} to get the needed upgrade:
\begin{lem}\label{l4.3}
Assume that there exists a  smooth function $\tilde{u}$ on $M$ such that for some $\tilde{\delta}_0>0$,
\begin{equation*}
f_{\infty}\big(\lambda[\chi+dd^c\tilde{u}]\big)\ge e^{G(x)+c}+\tilde{\delta}_0,\,\,\lambda[\chi+dd^c\tilde{u}]\in \Gamma_{\infty}.
\end{equation*}
Let $\varphi_j$ be the solution to (\ref{4.1NNNew}).  Assume that there is a subsequence $\{\varphi_{j_k}\}$ that converges in $L^1$.  Then this subsequence $\{\varphi_{j_k}\}$ converges uniformly.
\end{lem}
\begin{proof}
We wish to use Theorem \ref{t4.1N} below.  We put $\tilde{\chi}=\chi+dd^c\tilde{u}$,  then we have:
\begin{equation*}
\lambda[\tilde{\chi}]\in \Gamma_{\infty},\,\,f_{\infty}\big(\lambda[\tilde{\chi}]\big)\ge e^{G(x)+c}+\tilde{\delta}_0.
\end{equation*}
Since $G_j\rightarrow G$ uniformly and $c_j\rightarrow c$,  we see that,  with $k,\,l$ large enough,  one has:
\begin{equation*}
f_{\infty}\big(\lambda[\tilde{\chi}]\big)\ge e^{\max(G_{j_k}+c_{j_k},G_{j_l}+c_{j_l})}+\frac{1}{2}\tilde{\delta}_0.
\end{equation*}
Moreover,  if we put $\phi_1=\varphi_{j_k}-\tilde{u}$,  $\phi_2=\varphi_{j_l}-\tilde{u}$,  and we also denote $G_1=G_{j_k}+c_{j_k}$,  $G_2=G_{j_l}+c_{j_l}$,  then we know that $\phi_i,\,i=1,\,2$ satisfies:
\begin{equation*}
f\big(\lambda[\tilde{\chi}+dd^c\phi_i]\big)=e^{G_i},\,\,\lambda[\tilde{\chi}+dd^c\phi_i]\in \Gamma,\,\,\,i=1,\,2.
\end{equation*}
Now we are in the position to apply Theorem \ref{t4.1N} to get:
\begin{equation*}
||\varphi_{j_k}-\varphi_{j_l}||_{L^{\infty}}=||\phi_1-\phi_2||_{L^{\infty}}\le C\big(||e^{G_{j_k}+c_{j_k}}-e^{G_{j_l}+c_{j_l}}||_{L^{\infty}}+||\varphi_{j_k}-\varphi_{j_l}||_{L^1}^{\frac{\nu}{\nu+1}}\big).
\end{equation*}
This implies the uniform convergence.
\end{proof}

As a consequence,  we immediately see that:
\begin{cor}
Under the assumption of Lemma \ref{l4.3},  there exists a viscosity solution to:
\begin{equation*}
f\big(\lambda[\chi+dd^c\varphi]\big)=e^{G(x)+c},\,\,\lambda[\chi+dd^c\varphi]\in \Gamma.
\end{equation*}
Here $c$ is given by (\ref{c}).
\end{cor}

\subsection{Stability}
The crucial estimate we need for the existence proof is the following estimate:
\begin{thm}\label{t4.1N}
Let $G_1,\,G_2\in C^{\infty}(M)$.  Let $0<\delta_0<1$.  Assume that there exists a real smooth $(1,1)$ form $\tilde{\chi}$ on $M$ such that:
\begin{equation*}
\lambda[\tilde{\chi}]\in \Gamma_{\infty},\,\,f_{\infty}(\lambda[\tilde{\chi}])\ge e^{\max(G_1,G_2)}+\delta_0.
\end{equation*}
Let $\phi_i,\,i=1,\,2$ be the solution to the following equations:
\begin{equation*}
f\big(\lambda[\tilde{\chi}+dd^c\phi_i]\big)=e^{G_i},\,\,\,\lambda[\tilde{\chi}+dd^c\phi_i]\in \Gamma,\,\,i=1,\,2.
\end{equation*}
 Then for any $0<\nu<\frac{1}{n}$,  there exists a constant $C>0$,  such that:
\begin{equation*}
||\phi_1-\phi_2||_{L^{\infty}}\le C\big(||e^{G_1}-e^{G_2}||_{L^{\infty}}+||\phi_1-\phi_2||_{L^1}^{\frac{\nu}{1+\nu}}\big).
\end{equation*}
Here the constant $C$ depends on $\nu$,  an upper bound for $||G_1||_{L^{\infty}}$,  $||G_2||_{L^{\infty}}$,  structural conditions for $f$,  $\Gamma$ as well as $f_{\infty}$,  $\Gamma_{\infty}$,  the form $\tilde{\chi}$ as well as the background metric $\omega_0$.
\end{thm}

We will start with the following Moser-Trudinger type inequality.  

\begin{prop}\label{p4.1}
Let $0<\delta<1$.  
Let $G_0\in C(M)$.  
Let $\tilde{\chi}$ be a smooth $(1,1)$ form on $M$ such that for some $\delta_0>0$,
\begin{equation}
f_{\infty}\big(\lambda(\tilde{\chi})\big)\ge e^{G_0}+\delta_0,\,\,\,\lambda[\tilde{\chi}]\in \Gamma_{\infty}.
\end{equation}
Let $\phi,\,v\in C^{\infty}(M)$ be such that:
\begin{equation}\label{v-phi}
\begin{split}
&\tilde{\chi}+dd^c\phi\in \Gamma,\,\,f\big(\lambda[\tilde{\chi}+dd^c\phi]\big)\le e^{G_0(x)}+\delta_1,\\
&\tilde{\chi}+dd^cv\in \Gamma,\,\,f\big(\lambda[\tilde{\chi}+dd^cv]\big)\ge e^{G_0}-\delta_1,
\end{split}
\end{equation}
with $\delta_1$ small enough such that $\delta_1\delta<\frac{1}{2}\delta_0$.  
Let $\kappa>1$,  $s>0$,  we define:
\begin{equation}\label{4.7New}
A_{\delta,s,\kappa}=\big(\int_M[(-\phi+(1-\delta)v-s)_+]^{\kappa}e^{n\kappa G_0}\omega_0^n\big)^{\frac{1}{\kappa}}.
\end{equation}
Then there exists $c_0>0,\,C_0>0$ such that:
\begin{equation}\label{4.3NN}
\int_M\exp\big(c_0\frac{\big[((1-\delta)v-\phi-s)^+\big]^{\frac{n+1}{n}}}{\delta A_{\delta,s,\kappa}^{\frac{1}{n}}}\omega_0^n\le C_0\exp\big(C_0\frac{A_{\delta,s,\kappa}}{\delta}\big).
\end{equation}
In the above,  $\alpha_0$ and $C_0$ depends only on $\kappa$,  the background metric $\omega_0$,  $||G_0||_{L^{\infty}}$,  on the form $\tilde{\chi},\,\delta_0$ as well as the structural conditions on $f$ and $\Gamma$.
\end{prop}
The proof of the above proposition makes use of an auxiliary Monge-Ampère equation. 
Since $\lambda[\tilde{\chi}]\in \Gamma_{\infty}$,  there exists $\eps_0>0$ such that:
\begin{equation}
\lambda(\tilde{\chi}-\eps_0\omega_0)\in \Gamma_{\infty},\,\,\,f_{\infty}\big(\lambda[\tilde{\chi}-\eps_0\omega_0]\big)\ge e^{G_0}+\frac{2}{3}\delta_0.
\end{equation}
Note that this choice of $\eps_0$ depends only on the form $\tilde{\chi}$,  the background metric $\omega_0$ as well as the structural condition on $f_{\infty}$.

Let $\chi_j:\bR\rightarrow \bR_+$ be a sequence of smooth functions,  such that $\chi_j(x)\ge \max(x,0):=x_+$ and $\chi_j(x)$ decreasingly converges to $\max(x,0)$ as $j\rightarrow \infty$.

Let $\kappa>1,\,0<\delta<1,\,s\in \bR$,  
we define $\psi_{j,\delta,s,\kappa}$ to be the solution to the following Monge-Ampère equation:
\begin{equation}
(\eps_0\omega_0+dd^c\psi_{j,\delta,s,\kappa})^n
=e^{C_{j,\delta,s,\kappa}}\frac{\chi_j(-\phi+(1-\delta)v-s)}{A_{j,\delta,s,\kappa}}e^{nG_0}\omega_0^n,\,\,\sup_M\psi_{j,\delta,s,\kappa}=0.
\end{equation}
Here
\begin{equation*}
A_{j,\delta,s,\kappa}=\bigg(\int_M\chi_j^{\kappa}(-\phi+(1-\delta)v-s)e^{n\kappa G_0}\omega_0^n\bigg)^{\frac{1}{\kappa}}.
\end{equation*}
Proposition \ref{p4.1} would immediately follow from:
\begin{prop}\label{p4.2}
There exists universal constant $c_2>0,\,C_2>0$ such that 
\begin{equation*}
-\phi+(1-\delta)v-s\le c_2\delta^{\frac{n}{n+1}}A_{j,\delta,s,\kappa}^{\frac{1}{n+1}}(-\psi_{j,\delta,s,\kappa}+C_2\delta^{-1}A_{j,\delta,s,\kappa})^{\frac{n}{n+1}}.
\end{equation*}
Here $c_2$ and $C_2$ depends on $\kappa$,  $\tilde{\chi},\,\omega_0$,  a bound for $||G_0||_{L^{\infty}}$ as well as structural conditions of $f$ and $\Gamma$.
\end{prop}
\begin{proof}
We define
\begin{equation*}
\Psi=-\eps(-\psi_{j,\delta,s,\kappa}+\Lambda)^{\frac{n}{n+1}}+\big[-\phi+(1-\delta)v-s\big].
\end{equation*}
We wish to show that,  with suitable choice of $\eps,\,\Lambda$,  one has $\Psi\le 0$.  Assume that the function $\Psi$ achieves maximum at $x_{\max}\in M$,  then we may calculate at $x_{\max}$:
\begin{equation*}
\begin{split}
&0\ge dd^c\Psi=\frac{\eps n}{n+1}(-\psi_{j,\delta,s,\kappa}+\Lambda)^{-\frac{1}{n+1}}dd^c\psi_{j,\delta,s,\kappa}+\big(-dd^c\phi+(1-\delta)dd^cv\big)\\
&+\frac{\eps n}{(n+1)^2}(-\psi_{j.\delta,s,\kappa}+\Lambda)^{-\frac{n+2}{n+1}}d\psi_{j,\delta,s,\kappa}\wedge d^c\psi_{j,\delta,s,\kappa}\\
&\ge \frac{\eps n}{n+1}(-\psi_{j,\delta,s,\kappa}+\Lambda)^{-\frac{1}{n+1}}(\eps_0\omega_0+dd^c\psi_{j,\delta,s,\kappa})-(\eps_0\omega_0+dd^c\phi)+(1-\delta)(\eps_0\omega_0+dd^cv)\\
&+\eps_0\big(\delta-\frac{\eps n}{n+1}(-\psi_{j,\delta,s,\kappa}+\Lambda)^{-\frac{1}{n+1}}\big)\omega_0.
\end{split}
\end{equation*}
We wish to choose the constants $\eps,\,\Lambda$ so that:
\begin{equation}
\frac{1}{2}\delta =\frac{\eps n}{n+1}\Lambda^{-\frac{1}{n+1}}.
\end{equation}
With this choice,  and using that $\psi_{j,\delta,s,\kappa}\le 0$,  we see that:
\begin{equation}\label{4.6NNNN}
\eps_0\omega_0+dd^c\big(\frac{\phi}{\delta}-\frac{1-\delta}{\delta}v\big)> \frac{\eps n}{\delta(n+1)}(-\psi_{j,\delta,s,\kappa}+\Lambda)^{-\frac{1}{n+1}}(\eps_0\omega_0+dd^c\psi_{j,\delta,s,\kappa})\ge 0.
\end{equation}
Using the lemma \ref{l4.9new} below,  we know that there exists a constant $C_1>0$ such that at $x_{\max}$:
\begin{equation}\label{4.7NNNN}
\eps_0\omega_0+dd^c\big(\frac{\phi}{\delta}-\frac{1-\delta}{\delta}v\big)\le C_1\omega_0.
\end{equation}
Denote $\omega_2=\eps_0\omega_0+dd^c\big(\frac{\phi}{\delta}-\frac{1-\delta}{\delta}v\big)$,  then we take trace with respect to $\omega_2$ on both sides of (\ref{4.6NNNN}),  we see that at $x_{\max}$:
\begin{equation}\label{4.8NNN}
\begin{split}
&n\ge \frac{\eps n}{\delta(n+1)}(-\psi_{j,\delta,s,\kappa}+\Lambda)^{-\frac{1}{n+1}}tr_{\omega_2}\big(\eps_0\omega_0+dd^c\psi_{j,\delta,s,\kappa}\big)\\
&\ge \frac{\eps n}{\delta(n+1)}(-\psi_{j,\delta,s,\kappa}+\Lambda)^{-\frac{1}{n+1}}n\big(\frac{(\eps_0\omega_0+dd^c\psi_{\delta,s,\kappa})^n}{\omega_2^n}\big)^{\frac{1}{n}}\\
&\ge \frac{\eps n^2}{\delta(n+1)}(-\psi_{j,\delta,s,\kappa}+\Lambda)^{-\frac{1}{n+1}}\big(e^{C_{j,\delta,s,\kappa}}\frac{(-\phi+(1-\delta)v-s)_+}{A_{j,\delta,s,\kappa}}e^{nG_0}C_1^{-n}\big)^{\frac{1}{n}}
\end{split}
\end{equation}
In the second line above,  we used the arithmetic-geometric inequality.  In the third line above,  we used (\ref{4.7NNNN}) as well as the Monge-Ampère equation solved by $\psi_{j,\delta,s,\kappa}$.  Using the lemma below \ref{l4.8New},  we know that $C_{j,\delta,s,\kappa}$ is bounded from below depending only on the background metric $(M,\omega_0)$ and the choice of $\eps_0$.   Hence we obtain from (\ref{4.8NNN}) that there exists $c>0$,  such that 
\begin{equation*}
1\ge c\frac{\eps }{\delta}(-\psi_{\delta,s,\kappa}+\Lambda)^{-\frac{1}{n+1}}A_{\delta,s,\kappa}^{-\frac{1}{n}}(-\phi+(1-\delta)v-s)_+^{\frac{1}{n}}.
\end{equation*}
Here $c$ depends on the background metric $(M,\omega_0)$,  the constant $C_1>0$ as well as a lower bound for $G_0$.  Now we choose $\eps$ so that:
\begin{equation}
\eps=\frac{1}{c^n}\frac{\delta^n}{\eps^n}A_{\delta,s,\kappa}.
\end{equation}
With this choice,  we have:
\begin{equation*}
\eps\big(-\psi_{\delta,s,\kappa}+\Lambda\big)^{\frac{n}{n+1}}\ge \frac{1}{c^n}\frac{\delta^n}{\eps^n}A_{\delta,s,\kappa}\big(-\psi_{\delta,s,\kappa}+\Lambda\big)^{\frac{n}{n+1}}\ge \big(-\phi+(1-\delta)v-s\big)_+.
\end{equation*}
This implies $\Psi\le 0$ at $x_{\max}$.  
\end{proof}
In the above,  we used the following lemma:
\begin{lem}\label{l4.8New}
Let $G(x)$ be a smooth function on $M$.  Let $\varphi$ be the smooth solution to:
\begin{equation*}
(\eps_0\omega_0+dd^c\varphi)^n=
e^{G+c}\omega_0^n.
\end{equation*}
for some constant $c\in \bR$.  Let $p>1$,  then
the constant $c$ can be estimated from below by $p$,  $\eps_0$,  the background metric $\omega_0$ an an upper bound for $||e^G||_{L^p(\omega_0^n)}$.  
\end{lem}
The proof of this statement is contained in Lemma 5.9 of Kołodziej and Nguyen \cite{KN0}.

A crucial step in the above proof is (\ref{4.7NNNN}),  whose proof will be explained below:
\begin{lem}\label{l4.9new}
Under the assumptions of Proposition \ref{p4.1},  there exists a constant $C_1>0$,  such that at $x_{\max}$:
\begin{equation*}
\eps_0\omega_0+dd^c\big(\frac{\phi}{\delta}-\frac{1-\delta}{\delta}v\big)\le C_1\omega_0.
\end{equation*}
Here $C_1$ depends only on $||G_0||_{L^{\infty}}$,  $\tilde{\chi}$,  $\delta_0$ as well as structural conditions on $f$ and $\Gamma$.
\end{lem}

\begin{proof}
First we assume that $\lambda[\tilde{\chi}+dd^c\big(\frac{\phi}{\delta}-\frac{1-\delta}{\delta}v\big)]\in \Gamma$.  
Note that
\begin{equation*}
\tilde{\chi}+dd^c\phi=\delta\big(\tilde{\chi}+dd^c\big(\frac{\phi}{\delta}-\frac{1-\delta}{\delta}v\big)\big)+(1-\delta)\big(\tilde{\chi}+dd^cv\big).
\end{equation*}
Then we use the concavity of $A\mapsto F(A):=f(\lambda[A])$ in the space of Hermitian metrics,  we see that at $x_{\max}$:
\begin{equation*}
\begin{split}
&e^{G_0}+\delta_1\ge F(\tilde{\chi}+dd^c\phi)\ge \delta F(\tilde{\chi}+dd^c(\frac{\phi}{\delta}+\frac{1-\delta}{\delta}v))+(1-\delta)F(\tilde{\chi}+dd^cv)\\
&\ge \delta F\big(\tilde{\chi}+dd^c(\frac{\phi}{\delta}+\frac{1-\delta}{\delta} v)\big)+(1-\delta)(e^{G_0}-\delta_1).
\end{split}
\end{equation*}
Therefore
\begin{equation}\label{4.10NNN}
F(\tilde{\chi}+dd^c(\frac{\phi}{\delta}+\frac{1-\delta}{\delta}v))=f\big(\lambda[\tilde{\chi}+dd^c(\frac{\phi}{\delta}+\frac{1-\delta}{\delta}v)]\big)\le e^{G_0}+\frac{\delta_1}{\delta}(2-\delta)\le e^{G_0}+\frac{1}{2}\delta_0.
\end{equation}
On the other hand,  from our choice of $\eps_0$,  we know that:
\begin{equation}\label{4.11NNN}
F_{\infty}\big(\tilde{\chi}-\eps_0\omega_0\big)=f_{\infty}\big(\lambda[\tilde{\chi}-\eps_0\omega_0]\big)\ge e^{G_0}+\frac{2}{3}\delta_0.
\end{equation}
Define $\omega_2=\eps_0\omega_0+dd^c\big(\frac{\phi}{\delta}+\frac{1-\delta}{\delta}v\big)$,  then (\ref{4.10NNN}) is equivalent to:
\begin{equation}\label{4.12NNN}
F\big(\tilde{\chi}-\eps_0\omega_0+\omega_2\big)\le e^{G_0}+\frac{1}{2}\delta_0.
\end{equation}
Using Lemma \ref{l4.9N} below,  we can conclude that there exists a constant $C_{0.5}>0$ such that $\omega_2\le C_{0.5}\omega_0$,  since $\omega_2(x_{max})$ is positive by (\ref{4.6NNNN}).

If,  on the other hand,  $\lambda[\tilde{\chi}+dd^c(\frac{\phi}{\delta}-\frac{1-\delta}{\delta}v)]\notin \Gamma$.  Then we would have:
\begin{equation*}
\lambda[\tilde{\chi}-\eps_0\omega_0]\in \Gamma_{\infty},\,\,\lambda[\tilde{\chi}-\eps_0\omega_0+\omega_2]\notin \Gamma.
\end{equation*}
Again from Lemma \ref{l4.9N} below,  we can conclude that $\omega_2\le C_{0.5}\omega_0$.
\end{proof}
The above proof needs the following lemma:
\begin{lem}\label{l4.9N}
Let $\delta>0$.  Let $\tilde{\chi}$ be a smooth real $(1,1)$ form on $M$ such that $\lambda[\tilde{\chi}]\in \Gamma_{\infty}$ and $F_{\infty}(\tilde{\chi})\ge e^{G_0}+\delta$.  Then the following hold:
\begin{enumerate}
\item There exists $C_{0.5}>0$ such that for any $x_0\in M$ and any $(1,1)$ form $\omega$ at $x_0$ with $\lambda[\omega](x_0)\in \Gamma_n$ and $\lambda[\tilde{\chi}+\omega](x_0)\notin \Gamma$,  one has $\omega\le C_{0.5}\omega_0$ at $x_0$.
\item Fix any $0<\delta'<\delta$,  there exists a constant $C_{0.6}>0$ such that for any $x_0\in M$ and any $(1,1)$ form $\omega$ at $x_0$ with $\lambda[\omega]\in \Gamma_n$,  $\lambda[\tilde{\chi}+\omega](x_0)\in \Gamma$,  and $F(\tilde{\chi}+\omega)(x_0)\le e^{G_0}+\delta'$,  one has $\omega\le C_{0.6}\omega_0$.
\end{enumerate}
In the above,  $C_{0.5}$ and $C_{0.6}$ depend on the $(1,1)$ form $\tilde{\chi}$,  the cone $\Gamma_{\infty}$,  the functions $f$ and $f_{\infty}$.
\end{lem}
\begin{proof}
Fix any $x_0\in M$,  we can choose normal coordinate near $x_0$ so that $(\omega_0)_{i\bar{j}}(x_0)=\delta_{ij}$.  Under this coordinate,  we can write:
\begin{equation*}
\tilde{\chi}=\sqrt{-1}\tilde{\chi}_{i\bar{j}}dz_i\wedge d\bar{z}_j,\,\,\,\omega=\sqrt{-1}\omega_{i\bar{j}}dz_i\wedge d\bar{z}_j.
\end{equation*}
That $\lambda[\omega](x_0)\in \Gamma_n$ simply means $\omega_{i\bar{j}}(x_0)$ is nonnegative definite.  From the variational characterization of eigenvalues,  we have:
\begin{equation}\label{4.13NN}
\begin{split}
&\lambda_l[\tilde{\chi}+\omega](x_0)=\inf_{\dim \Sigma=l}\sup_{v\in \Sigma,\,|v|=1}(\tilde{\chi}_{i\bar{j}}(x_0)+\omega_{i\bar{j}}(x_0))v_{\bar{i}}v_j\\
&\ge \inf_{\dim \Sigma=l}\sup_{v\in \Sigma,\,|v|=1}\tilde{\chi}_{i\bar{j}}(x_0)v_{\bar{i}}v_j=\lambda_l[\tilde{\chi}](x_0),\,\,1\le l\le n.
\end{split}
\end{equation}
In the above,  inf is taken over all $\bC$-subspaces of dimension $l$,  and we have listed the eigenvalues in the increasing order.  Therefore,  for any positive real $(1,1)$ form $\omega$ at $x_0$,  there exists $\lambda'\in \Gamma_n$ such that:
\begin{equation*}
\lambda[\tilde{\chi}+\omega](x_0)=\lambda[\tilde{\chi}](x_0)+\lambda'.
\end{equation*}
On the other hand,  since $M$ is compact,  $\lambda[\tilde{\chi}](x_0),\,x_0\in M$ should remain in a compact subset of $\Gamma_{\infty}$,  we see from (\ref{4.13NN}) that:
\begin{equation}\label{4.14NN}
|\lambda_l[\tilde{\chi}+\omega](x_0)-\lambda_l[\omega](x_0)|\le C,\,\,\,1\le l\le n.
\end{equation}
where $C$ depends only on the form $\tilde{\chi}$ and not on $\omega$ or $x_0$.  Since $\lambda[\tilde{\chi}](x_0)$ remains in a compact subset of $\Gamma_{\infty}$ when $x_0$ ranges over $M$,  there exists $C'>0$ such that for any $\lambda'\in \Gamma_n$ with $|\lambda'|\ge C'$ one would have $\lambda[\tilde{\chi}](x_0)+\lambda'\in \Gamma$ from the first part of Lemma \ref{l2.9New}. This fact combined with (\ref{4.14NN}) implies that if $\lambda[\omega](x_0)\in \Gamma_n$ with $|\lambda[\omega](x_0)|\ge C''$ and $C''$ sufficiently large,  we would have $\lambda[\tilde{\chi}+\omega](x_0)\in \Gamma$.  Hence if we require that $\lambda[\tilde{\chi}+\omega](x_0)\notin \Gamma$,  $\lambda[\omega](x_0)$ must remains in a bounded set,  giving us the desired conclusion for part (1).

For part (2),  our situation now is that:
\begin{equation*}
f\big(\lambda(\tilde{\chi})(x_0)+\lambda'\big)\le e^{G_0}+\delta',\,\,f_{\infty}(\lambda(\tilde{\chi})(x_0))\ge e^{G_0}+\delta,
\end{equation*}
where $\lambda'=\lambda[\tilde{\chi}+\omega](x_0)-\lambda[\tilde{\chi}](x_0)$ and we know $\lambda'\in \Gamma_n$.   Using the second part of Lemma \ref{l2.9New},  since $\lambda[\tilde{\chi}](x_0)$ remains in a compact subset of $\Gamma_{\infty}$,  there exists $C''>0$ such that for any $\lambda'\in \Gamma_n$,  $|\lambda'|\ge C''$,  one has:
\begin{equation*}
\lambda[\tilde{\chi}](x_0)+\lambda'\in \Gamma,\,\,f\big(\lambda[\tilde{\chi}](x_0)+\lambda')\ge e^{G_0}+\frac{\delta+\delta'}{2}.
\end{equation*}
Therefore,  we must have $|\lambda[\tilde{\chi}+\omega](x_0)-\lambda[\tilde{\chi}](x_0)|\le \tilde{C}''$.  Hence $|\lambda[\omega](x_0)|\le C+C''$,  therefore,  $\omega$ remains in a bounded set.
\end{proof}

As a consequence of Proposition \ref{p4.2},  we can finally prove Proposition \ref{p4.1}.
\begin{proof}
(Of Proposition \ref{p4.1}) Proposition \ref{p4.2} gives:
\begin{equation*}
\frac{(-\phi-(1-\delta)v-s)_+^{\frac{n+1}{n}}}{c_2^{\frac{n+1}{n}}\delta A_{j,\delta,s,\kappa}^{\frac{1}{n}}}\le -\psi_{j,\delta,s,\kappa}+C_2\delta^{-1}A_{j,\delta,s,\kappa}.
\end{equation*}
Since for each $j$,  $\psi_{j,\delta,s,\kappa}$ is an $\omega_0$-psh function with $\sup_M\psi_{j,\delta,s,\kappa}$,  there exists $C>0,\,\alpha_0>0$ which depends only on $(M,\omega_0)$ such that $\int_Me^{-\alpha_0\psi_{j,\delta,s,\kappa}}\omega_0^n\le C$.  Therefore:
\begin{equation*}
\begin{split}
&\int_M\exp\big(\alpha_0\frac{(-\phi+(1-\delta)v-s)_+^{\frac{n+1}{n}}}{c_2^{\frac{n+1}{n}}\delta A_{j,\delta,s,\kappa}^{\frac{1}{n}}}\big)\omega_0^n\le \int_Me^{-\alpha_0\psi_{j,\delta,s,\kappa}}\omega_0^n\cdot \exp\big(C_2\delta^{-1}A_{j,\delta,s,\kappa}\big)\\
&\le C\exp\big(C_2\delta^{-1}A_{j,\delta,s,\kappa}\big).
\end{split}
\end{equation*}
Now we pass to limit as $j\rightarrow \infty$,  and note that $A_{j,\delta,s,\kappa}$ converges to $A_{\delta,s,\kappa}$ given by (\ref{4.7New}).
\end{proof}

\begin{lem}
Assume that $v,\,\phi$ satisfies (\ref{v-phi}).  
Assume that $s_0,\,\delta>0$ are chosen so that $A_{\delta,s_0,\kappa}\le \delta$.  Let $p_0>1$.  Define $\Omega_{\delta,s}=\{x\in M:(1-\delta)v-\phi-s>0\}$.  Put $u(s)=\int_{\Omega_{\delta,s}}e^{nG_0}\omega_0^n$.  Then for any $0<\delta_*<\frac{1}{n}$,  one can choose $\kappa$ sufficiently close to 1,  such that there exists $C_*>0$,  depending on $\delta_*$,  $||e^{nG_0}||_{L^{p_0}(\omega_0^n)}$,  the background metric and the choice of $\kappa$ such that for any $s\ge s_0,\,t>0$,  one has:
\begin{equation*}
tu(s+t)\le C_*\delta u(s)^{1+\delta_*}.
\end{equation*}
\end{lem}
\begin{proof}
Since $A_{\delta,s,\kappa}$ is monotone decreasing in $s$,  we see that:
\begin{equation}
A_{\delta,s,\kappa}\le \delta,\,\,s\ge s_0.
\end{equation}
From (\ref{4.3NN}),  we see that for any $q\ge 1$,
\begin{equation}\label{4.16New}
\int_M(1-\delta)v-\phi-s)_+^{q\frac{n+1}{n}}\omega_0^n\le C_2(q)\delta^qA_{\delta,s,k}^{\frac{q}{n}},\,\,s\ge s_0.
\end{equation}
Therefore,
\begin{equation*}
\begin{split}
&A_{\delta,s,\kappa}=||((1-\delta)v-\phi-s)_+e^{nG_0}||_{L^{\kappa}(\omega_0^n)}\le ||((1-\delta)v-\phi-s)_+||_{L^{\frac{n+1}{n}q}(\omega_0^n)}||e^{nG_0}||_{L^{q'}(\Omega_{\delta,s},\omega_0^n)}\\
&\le C_2(q)^{\frac{n}{q(n+1)}}\delta^{\frac{n}{n+1}}A_{\delta,s,\kappa}^{\frac{1}{n+1}}||e^{nG_0}||_{L^{q'}(\Omega_{\delta,s},\omega_0^n)}\le C_3(q)\delta^{\frac{n}{n+1}}A_{\delta,s,\kappa}^{\frac{1}{n+1}}||e^{nG_0}||^{\lambda}_{L^1(\Omega_{\delta,s},\omega_0^n)}||e^{nG_0}||_{L^{p_0}}^{1-\lambda}.
\end{split}
\end{equation*}
In the above,  $q'$ is chosen such that $\frac{1}{\kappa}=\frac{n}{(n+1)q}+\frac{1}{q'}$.  The second inequality used (\ref{4.16New}),  and in the last inequality,  $\lambda$ satisfies: $\frac{1}{q'}=\lambda+\frac{1-\lambda}{p_0}$.  Therefore we get:
\begin{equation}\label{4.17NN}
A_{\delta,s,\kappa}\le C_4(q)\delta||e^{nG_0}||_{L^{p_0}}^{\frac{n+1}{n}(1-\lambda)} ||e^{nG_0}||_{L^1(\Omega_{\delta,s},\omega_0^n)}^{\lambda\frac{n+1}{n}}.
\end{equation}
On the other hand,  for any $t>0$,  
\begin{equation}\label{4.18NN}
\begin{split}
&A_{\delta,s,\kappa}=||\big((1-\delta)v-\phi-s\big)_+e^{nG_0}||_{L^{\kappa}(\omega_0^n)}\ge ||\big((1-\delta)v-\phi-s\big)_+e^{nG_0}||_{L^1(\omega_0^n)}\\
&\ge t||e^{nG_0}||_{L^1(\Omega_{s+t},\omega_0^n)}.
\end{split}
\end{equation}
Combining (\ref{4.17NN}) and (\ref{4.18NN}) we see that 
\begin{equation}
t||e^{nG_0}||_{L^1(\Omega_{s+t},\omega_0^n)}\le C_5\delta||e^{nG_0}||_{L^1(\Omega_{\delta,s},\omega_0^n)}^{\lambda\frac{n+1}{n}}.
\end{equation}
Here we note that by choosing $\kappa$ sufficiently close to 1,  one can make $\lambda \frac{n+1}{n}$ as close to $\frac{n+1}{n}$ as one wants.  Indeed,  by choosing $\kappa$ sufficiently close to 1 and $q$ sufficiently large,  we can make $q'$ as close to 1 as one desires.  With $q'$ sufficiently close to 1,  one has that $\lambda$ can be made as close to 1 as one desires.
\end{proof}

We need to use the following Lemma of De Giorgi which was first used in the setting
of complex Monge-Ampère equations in Kołodziej\cite{K}:
\begin{lem}\label{l4.6}
Let $u:[0,+\infty)\rightarrow [0,+\infty)$ be a decreasing function,  such that there exists $\mu>0,\,B_0>0,\,s_0\ge 0$,  such that for any $t>0,\,s\ge s_0$,  one has:
\begin{equation*}
tu(s+t)\le B_0u(s)^{1+\mu}.
\end{equation*}
Then $u(s)\equiv 0$ for $s\ge s_0+\frac{2B_0u(s_0)^{\mu}}{1-2^{-\mu}}.$
\end{lem}
\begin{proof}
We can choose a sequence $\{s_k\}_{k\ge 1}$ by induction:
\begin{equation*}
s_{k+1}-s_k=2B_0u(s_k)^{\mu}.
\end{equation*}
Then we choose $s=s_k,\,t=s_{k+1}-s_k$ we see that:
\begin{equation*}
u(s_{k+1})\le \frac{B_0u(s_k)^{1+\mu}}{s_{k+1}-s_k}\le \frac{1}{2}u(s_k).
\end{equation*}
That is,  $u(s_k)\le 2^{-k}u(s_0)$.  Hence:
\begin{equation*}
s_{k+1}-s_k\le 2B_0u(s_0)^{\mu}2^{-k\mu}.
\end{equation*}
Therefore,
\begin{equation*}
\sup_{k\ge 0}s_k=s_0+\sum_{k\ge 0}(s_{k+1}-s_k)\le s_0+\frac{2B_0u(s_0)^{\mu}}{1-2^{-\mu}}.
\end{equation*}
Since $u(s)$ is decreasing,  we see that $u(s)\equiv 0$ for $s\ge \sup_{k\ge 0}s_k$.
\end{proof}
As a consequence,  we see that:
\begin{cor}\label{c4.7}
Let $\frac{2\delta_1}{\delta_0}\le \delta<1$,  $s_0>0$ be chosen so that $A_{\delta,s,\kappa}\le \delta$.  Then for any $0<\nu<\frac{1}{n}(1-\frac{1}{p_0})$,  we may choose $\kappa>1$ sufficiently close to 1,  such that there exists $C>0$ depending only on $\nu$,  $||e^{nG_0}||_{L^{p_0}}$,  the background metric and the choice of $\kappa$ such that:
\begin{equation*}
\sup_M\big((1-\delta)v-\phi)\le s_0+C\delta vol(\Omega_{\delta,s_0})^{\nu}.
\end{equation*}
\end{cor}
\begin{proof}
We wish to use Lemma \ref{l4.6} with $u(s)=\int_{\Omega_{\delta,s}}e^{nG_0}\omega_0^n$ with $B_0=C_*\delta$ and $\mu=\delta_*$.  Then Lemma \ref{l4.6} gives that $u(s)\equiv 0$ for $s\ge s_0+\frac{2C_*\delta u(s_0)^{\delta_*}}{1-2^{-\delta_*}}$.
This is equivalent to:
\begin{equation*}
\sup_M\big((1-\delta)v-\phi\big)\le s_0+\frac{2C_*}{1-2^{-\delta_*}}\delta \bigg(\int_{\Omega_{\delta,s_0}}e^{nG_0}\omega_0^n\bigg)^{\delta_*}.
\end{equation*}
On the other hand,
\begin{equation*}
\int_{\Omega_{\delta,s_0}}e^{nG_0}\omega_0^n\le\bigg( \int_{\Omega_{\delta,s_0}}e^{p_0nG_0}\omega_0^n\bigg)^{\frac{1}{p_0}}vol(\Omega_{\delta,s_0})^{1-\frac{1}{p_0}}
\end{equation*}
Hence the result immediately follows.
\end{proof}
With some more effort,  we can get rid of $\delta$ in the above:
\begin{lem}\label{l4.8}
Assume that $v$ is bounded.  Let $\frac{2\delta_1}{\delta_0}\le \delta<1$ and $s_0>0$ are chosen so that:
\begin{enumerate}
\item $s_0\ge 2\delta||v||_{L^{\infty}}$,
\item $A_{\delta,s_0,\kappa}\le \delta$. 
\end{enumerate}
Then for any $0<\nu<\frac{1}{n}(1-\frac{1}{p_0})$,  there exists $C>0$ depending only on $\nu$,  $||e^{nG_0}||_{L^{p_0}}$,  and the background metric,  such that:
\begin{equation*}
\sup_M(v-\phi)\le \frac{3s_0}{2}+Cs_0^{-\nu}||(v-\phi)_+||_{L^1}^{\nu}.
\end{equation*}
\end{lem}
\begin{proof}
From Corollary \ref{c4.7},  we see that
\begin{equation*}
\begin{split}
&\sup_M(v-\phi)\le \sup_M((1-\delta)v-\phi)+\delta||v||_{L^{\infty}}\le s_0+\delta||v||_{L^{\infty}}+C\delta vol(\Omega_{\delta,s_0})^{\nu}\\
&\le \frac{3s_0}{2}+C\delta vol(\Omega_{\delta,s_0})^{\nu}.
\end{split}
\end{equation*}
On the set $(1-\delta)v-\phi-s_0\ge 0$,  we see that $v-\phi\ge s_0-\delta||v||_{L^{\infty}}\ge \frac{s_0}{2}$.  Therefore one has:
\begin{equation*}
vol(\Omega_{\delta,s_0})\le \frac{2}{s_0}||(v-\phi)_+||_{L^1}.
\end{equation*}
Therefore the result follows.
\end{proof}
For $0<\delta<1$,  we define $s_0(\delta)$ to be the smallest $s_0$ such that $s_0\ge 2\delta||v||_{L^{\infty}}$ and $A_{\delta,s_0,\kappa}\le \delta$.  The following lemma gives the needed estimate for $s_0(\delta)$:
\begin{lem}\label{l4.9}
Assume that $\frac{2\delta_1}{\delta_0}\le \delta<1$.  There exists $C>0$ depending only on $||e^{nG_0}||_{L^{p_0}}$,  and the background metric,  such that:
\begin{equation*}
s_0(\delta)\le \max\big(2\delta||v||_{L^{\infty}},\,\,C||(v-\phi)_+||_{L^1}\big).
\end{equation*}
\end{lem}
\begin{proof}
With $s_0=s_0(\delta)$,  we have that either $s_0=2\delta||v||_{L^{\infty}}$ or $A_{s_0,\delta,\kappa}=\delta$.  If the first possibility holds,  then we are done.  If we are in the second possibility,  then
we can calculate,  with $\kappa<\beta<p_0$:
\begin{equation*}
\begin{split}
&A_{\delta,s_0,\kappa}=\bigg(\int_{\Omega_{\delta,s_0}}\big((1-\delta)v-\phi-s_0\big)_+^{\kappa}e^{n\kappa G_0}\omega_0^n\bigg)^{\frac{1}{\kappa}}\le ||((1-\delta)v-\phi-s)_+||_{L^{\frac{\kappa\beta}{\beta-\kappa}}(\omega_0^n)}||e^{nG_0}||_{L^{\beta}(\Omega_{\delta,s_0})}\\
&\le C(\beta)A_{\delta,s_0,\kappa}^{\frac{1}{n+1}}\delta^{\frac{n}{n+1}}||e^{nG_0}||_{L^{\beta}(\Omega_{\delta,s},\omega_0^n)}\le C(\beta)A_{\delta,s_0,\kappa}^{\frac{1}{n+1}}\delta^{\frac{n}{n+1}}||e^{nG_0}||_{L^{p_0}}vol(\Omega_{\delta,s_0})^{\frac{1}{\beta}-\frac{1}{p_0}}.
\end{split}
\end{equation*}
Using that $A_{\delta,s_0,\kappa}=\delta$,  and that $vol(\Omega_{\delta,s_0})\le \frac{2}{s_0}||(v-\varphi)_+||_{L^1}$,  we see that:
\begin{equation*}
1\le C(s_0^{-1})^{\frac{1}{\beta}-\frac{1}{p_0}}||(v-\phi)_+||_{L^1}^{\frac{1}{\beta}-\frac{1}{p_0}}.
\end{equation*}
So we are done.
\end{proof}
As a corollary,  we can show that:
\begin{cor}\label{c4.10}
Let $\tilde{\chi}$ be as assumed in Proposition \ref{p4.1} for some $\delta_0>0$.  
Let $\delta_1>0$.  Let $G_0\in C^{\infty}(M)$.  Let $v,\,\phi$ be smooth functions on $M$ that satisfy:
\begin{equation*}
\begin{split}
&\lambda[\tilde{\chi}+dd^c\phi]\in \Gamma,\,\,f\big(\lambda[\tilde{\chi}+dd^c\phi]\big)\le e^{G_0}+\delta_1,\\
&\lambda[\tilde{\chi}+dd^cv]\in \Gamma,\,\,f\big(\lambda[\tilde{\chi}+dd^cv]\big)\ge e^{G_0}-\delta_1.
\end{split}
\end{equation*}
Let $p_0>1$,  then for any $0<\nu<\frac{1}{n}(1-\frac{1}{p_0})$,  there exists a constant $C>0$,  such that
\begin{equation*}
\sup_M(v-\phi)\le C\big(\delta_1+||(v-\phi)_+||_{L^1}^{\frac{\nu}{\nu+1}}\big).
\end{equation*}
\end{cor}
\begin{proof}
Without loss of generality,  we may assume that $||(v-\phi)_+||_{L^1}\le 1$.  We wish to take $\delta=\frac{2\delta_1}{\delta_0}$ in Lemma \ref{l4.8},  and we wish to choose a suitable $s_0\ge s_0(\delta)$ with $\delta=\frac{2\delta_1}{\delta_0}$.  There are two cases to consider:

The first case is when $\frac{4\delta_1}{\delta_0}||v||_{L^{\infty}}\le C||(v-\phi)_+||_{L^1}^{\frac{\nu}{\nu+1}}$,  where $C$ is the constant given by Lemma \ref{l4.9}.  Then we can choose $s_0=C||(v-\phi)_+||_{L^1}^{\frac{\nu}{\nu+1}}$.  Indeed,  this choice satisfies $s_0\ge s_0(\delta)$ with $\delta=\frac{2\delta_1}{\delta_0}$,   by Lemma \ref{l4.9}.  With this choice,  we obtain from Lemma \ref{l4.8} that:
\begin{equation*}
\sup_M(v-\phi)\le C'||(v-\phi)_+||_{L^1}^{\frac{\nu}{\nu+1}}.
\end{equation*}

The other case is when $\frac{4\delta_1}{\delta_0}||v||_{L^{\infty}}\ge C||(v-\phi)_+||_{L^1}^{\frac{\nu}{\nu+1}}$.  Note that since $||(v-\phi)_+||_{L^1}\le 1$,  we know that $\frac{4\delta_1}{\delta_0}||v||_{L^{\infty}}\ge C||(v-\phi)_+||_{L^1}$.  In this case,  we will have to choose $s_0=\frac{4\delta_1}{\delta_0}||v||_{L^{\infty}}=2\delta||v||_{L^{\infty}}$.  This choice also satisfies Lemma \ref{l4.9}.  With this choice,  we go back to Lemma \ref{l4.8} and found that:
\begin{equation*}
s_0^{-\nu}||(v-\phi)_+||_{L^1}^{\nu}\le C'||(v-\phi)_+||_{L^1}^{\frac{\nu}{\nu+1}}.
\end{equation*} 
On the other hand,  $\frac{3s_0}{2}=\frac{6\delta_1}{\delta_0}||v||_{L^{\infty}}\le C''\delta_1.$ Hence the result follows.
\end{proof}
Finally we are ready to prove Theorem \ref{t4.1N}:
\begin{proof}
(of Theorem \ref{t4.1N})
We define the function $G_0$ so that $e^{G_0}=\frac{1}{2}e^{G_1}+\frac{1}{2}e^{G_2}$.  We also put $\delta_1=\frac{1}{2}||e^{G_1}-e^{G_2}||_{L^{\infty}}$.  Then we have:
\begin{equation*}
\begin{split}
&e^{G_0}-\delta_1\le f\big(\lambda[\tilde{\chi}+dd^c\phi_1]\big)=e^{G_1}\le e^{G_0}+\delta_1,\,\,\,\lambda[\tilde{\chi}+dd^c\phi_1]\in \Gamma,\\
&e^{G_0}-\delta_1\le f\big(\lambda[\tilde{\chi}+dd^c\phi_2]\big)=e^{G_2}\le e^{G_0}+\delta_1,\,\,\lambda[\tilde{\chi}+dd^c\phi_2]\in \Gamma.
\end{split}
\end{equation*}
Now we apply Corollary \ref{c4.10} with $v=\varphi_1,\,\phi=\varphi_2$,  we get:
\begin{equation*}
\sup_M(\phi_1-\phi_2)\le C(\delta_1+||(\phi_1-\phi_2)_+||_{L^1}^{\frac{\nu}{\nu+1}}).
\end{equation*}
Now we switch and let $v=\phi_2,\,\,\phi=\phi_1$,  then we get that:
\begin{equation*}
||\phi_1-\phi_2||_{L^{\infty}}\le C'\big(||e^{G_1}-e^{G_2}||_{L^{\infty}}+||\phi_1-\phi_2||_{L^1}^{\frac{\nu}{\nu+1}}\big),
\end{equation*} 
where $0<\nu<\frac{1}{n}(1-\frac{1}{p_0})$.  On the other hand,  since we are already assuming an $L^{\infty}$ bound for $G_1,\,G_2$,  we can take $p_0$ sufficiently large so that the range of $\nu$ can be taken to be any value in $0<\nu<\frac{1}{n}$.
\end{proof}

\section{Uniqueness of viscosity solution}
The uniqueness of the viscosity solution essentially follows the uniqueness result in our previous paper \cite{CX}. The only problem is that in our previous paper \cite{CX}, we additionally assume that $f$ is of homogeneous degree 1. This is used in the proof of the uniqueness result in  \cite{CX}. However, this can be replaced by the Lemma \ref{lem 2.3} in the current setting. For the convenience of readers, we sketch the proof of the uniqueness. The basic strategy to prove the uniqueness of viscosity solutions will be to perform sup/inf convolution of the solution so that we get a sub/super solution that is twice differentiable a.e. 
which implies the pointwise differential inequality wherever the solution is twice differentiable.   The reason we no longer consider Richberg technique is that it does not necessarily regularize a viscosity supersolution to a smooth supersolution.

Let $\varphi \in C(M)$ be a viscosity solution to $F(\chi+dd^c \varphi)= e^{G(x,\varphi)}.$ We define the sup convolution of $\varphi$ as follows:
\begin{equation}\label{5.1}
\varphi^{\eps}(z)= \sup_{\xi \in T_z M}(\varphi(\exp_z(\xi))+\eps -\frac{1}{\eps}|\xi|_z^2),
\end{equation}
where $\exp_z(\xi)$ is the exponential map at $z$, defined using the metric $\omega_0$, and $|\xi|_z$ denotes the length of the tangent vector $\xi$, again using the metric $\omega_0$. Similarly we define the inf convolution of $\varphi$:
\begin{equation}\label{5.2}
\varphi_{\eps}(z) =\inf_{\xi \in T_z M} (\varphi(\exp_z(\xi))-\eps +\frac{1}{\eps}|\xi|_z^2).
\end{equation}
The main reasons for considering (\ref{5.1}) and (\ref{5.2}) are that they are semi-convex/concave,  and they are sub/super solutions respectively.  More precisely,  we observe that: 
\begin{lem}
Let $\varphi\in C(M)$,and we define $\varphi^{\eps},\,\varphi_{\eps}$ according to (\ref{5.1}) and (\ref{5.2}).  Let $x_0\in M$ and we choose local coordinates in a neighborhood of $x_0$.  Then there exists a neighborhood $U$ of $x_0$,  and $C_{\eps}>0$,  such that: $z\mapsto \varphi^{\eps}(z)+C_{\eps}|z|^2$ is convex on $U$ under the coordinates,  and $z\mapsto \varphi_{\eps}(z)-C_{\eps}|z|^2$ is concave on $U$ under the coordinates.  In particular,  $\varphi^{\eps}$,  $\varphi_{\eps}$ are punctually second order differentiable a.e.  
\end{lem}
The proof of this lemma is contained in \cite{CX},  Lemma 5.14.

Next we observe that $\varphi^{\eps}$ and $\varphi_{\eps}$ defined above produce viscosity subsolution and supersolution, up to a small error. We denote 
\begin{equation*}
\rho_{\varphi}(r)=\max\{|\varphi(x)-\varphi(y)|:\,\,d_g(x,y)\le r,\,x,y\in M\}.
\end{equation*}

\begin{prop}\label{p4.3}
Let $\varphi$ be a viscosity solution to $F(\chi+dd^c \varphi)=e^{G(x,\varphi)}$ with $G(x,\varphi)$ continuous.  Then there exist continuous functions $\rho(\eps):(0,1)\rightarrow \bR_{>0}$ with $\rho(0+)=0$,  and $\rho_1(\eps,a_1),\,\rho_2(\eps,a_2):(0,1)^2\rightarrow \bR_{>0}$ with $\rho_i(0+,0+)=0,\,\,i=1,2$,  such that for any $0<\eps<1$,  $0<a_i<1,\,i=1,\,2$:
\begin{enumerate}
\item $\frac{\varphi^{\eps}}{1+a_1}$ is $\Gamma$-subharmonic with respect to $\frac{\chi+\rho(\eps)\omega_0}{1+a_1}$ in the viscosity sense,  \sloppy and $F(\frac{\chi+\rho(\eps)\omega_0}{1+a_1}+dd^c\frac{\varphi^{\eps}}{1+a_1})\ge e^{G(x,\frac{\varphi^{\eps}}{1+a_1})}-\rho_1(\eps,a_1)$ in the viscosity sense.
\item $\frac{\varphi_{\eps}}{1-a_2}$ satisfies $F(\frac{\chi-\rho(\eps)\omega_0}{1-a_2}+dd^c\frac{\varphi_{\eps}}{1-a_2})\le e^{G(x,\frac{\varphi_{\eps}}{1-a_2})}+\rho_2(\eps,a_2)$ in the viscosity sense.
\end{enumerate}
Here the functions $\rho,\,\rho_i,\,i=1,\,2$ are determined by $\rho_{\varphi}$,  $||\varphi||_{L^{\infty}}$,  the form $\chi$ and the background metric. 
\end{prop}
\begin{proof}
The proof is almost identical to the proof of Proposition 5.6 in Cheng-Xu \cite{CX}.  Except that there are several places in the proof of Proposition 5.6 in \cite{CX},  we used that $f$ is homogeneous of degree 1 which we no longer assume.  As we will explain,  the argument in these places can easily go through with the help of Lemma \ref{lem 2.3}.

The first place in \cite{CX} where we used that $f$ is homogeneous degree 1 is in formula (5.14),  where we are in the middle of the proof of the first part and concluded that:
\begin{equation}\label{5.3N}
f\big(\lambda[\chi+\rho(\eps)\omega_0+(1+a_1)dd^cP]\big)=(1+a_1)f\big(\lambda[\frac{\chi+\rho(\eps)\omega_0}{1+a_1}+dd^cP]\big).
\end{equation}
Here $a_1>0$,  and $\rho:\bR_+\rightarrow \bR_+$ is a quantity that $\rho(0+)=0$.  $P$ is a $C^2$ function that touches $\frac{\varphi^{\eps}}{1+a_1}$ from above at $x_0$.  If one no longer assumes $f$ is homogeneous degree 1,  then one can use Lemma \ref{lem 2.3} to see that one has ``$\le$" holds in (\ref{5.3N}),  and the subsequent argument is not affected.

The second place in \cite{CX} where we used that $f$ is homogeneous degree 1 is in the formula following (5.20),  where we are in the middle of the proof of the second part and the situation is completely the same,  where we need that for some $0<a_2<1$,  one has:
\begin{equation*}
(1-a_2)f(\lambda)\le f((1-a_2)\lambda),\,\,\,\lambda\in \Gamma.
\end{equation*}
\end{proof}
\subsection{when the right hand side has strict monotonicity}
In this subsection,  we assume that $G(x,\varphi)$ is continuous and strictly monotone increasing in $\varphi$.  We wish to show that:
\begin{thm}\label{t4.1}
Let $G(x,\varphi)$ be continuous,  and strictly monotone increasing in $\varphi$.  Then there exists at most one viscosity solution to $F(\chi+dd^c\varphi)=e^{G(x,\varphi)}$.
\end{thm}
If not,  then there exist two viscosity solutions $\varphi_1,\,\varphi_2$,  and $\varphi_1\neq \varphi_2$.  Without loss of generality,  we may assume that:
\begin{equation}
\kappa_0:=\max_M(\varphi_1-\varphi_2)>0.
\end{equation}
Now we consider the super-convolution,  applied to $\varphi_1$,  and the inf-convolution,  applied to $\varphi_2$.  We define:
\begin{equation*}
\kappa_{\eps,a_1,a_2}=\max_M\big(\frac{(\varphi_1)^{\eps}}{1+a_1}-\frac{(\varphi_2)_{\eps}}{1-a_2}\big).
\end{equation*}
Then it is easy to see that,  as $\eps,\,a_1,\,a_2\rightarrow 0+$,  $\kappa_{\eps,a_1,a_2}\rightarrow \kappa_0$.  Assume that the above max is achieved at $x_*$.  Then we see that the sub/super solution condition for $\varphi^{\eps}/\varphi_{\eps}$ holds in the classical sense at $x_*$,  provided that both of them are punctually second order differentiable at $x_*$:
\begin{prop}\label{p4.12}
Assume that $\eps,\,a_1,\,a_2$ are chosen so that $\frac{c_*+\rho(\eps)}{1+a_1}<\frac{c_*-\rho(\eps)}{1-a_2}$.  Assume also that both $(\varphi_1)^{\eps}$ and $(\varphi_2)_{\eps}$ are punctually second order differentiable at $x_*$. Then one has:
\begin{enumerate}
\item $\lambda[\frac{\chi-\rho(\eps)\omega_0}{1-a_2}+dd^c\frac{(\varphi_2)_{\eps}}{1-a_2}](x_*)\in \Gamma$,  
\item $F(\frac{\chi+\rho(\eps)\omega_0}{1+a_1}+dd^c\frac{(\varphi_1)^{\eps}}{1+a_1})(x_*)\le F(\frac{\chi-\rho(\eps)\omega_0}{1-a_2}+dd^c\frac{(\varphi_2)_{\eps}}{1-a_2})(x_*)$.
\end{enumerate}
In the above,  $c_*>0$ is some constant such that  $\lambda(\chi-c_*\omega_0)\in \Gamma$.
\end{prop}
This proposition is exactly the proposition 5.15 in \cite{CX},  and a proof is given there.  With the help of Proposition \ref{p4.3} and \ref{p4.12},  we have the uniqueness proof when the right hand side has strictly increasing dependence on $\varphi$,  at least when both $(\varphi_1)^{\eps}$ and $(\varphi_2)_{\eps}$ are both punctually second order differentiable at $x_*$:
\begin{proof}
(of Theorem \ref{t4.1},  assuming $(\varphi_1)^{\eps}$ and $(\varphi_2)_{\eps}$ are punctually second order differentiable at $x_*$)
We choose $\eps,\,a_1,\,a_2$ so that $\frac{c_*+\rho(\eps)}{1+a_1}<\frac{c_*-\rho(\eps)}{1-a_2}$.  
Combining Proposition \ref{p4.3} and Proposition \ref{p4.12},  we see that at $x_*$:
\begin{equation}
\begin{split}
&f\big(\lambda[\frac{\chi+\rho(\eps)\omega_0}{1+a_1}+dd^c\frac{(\varphi_1)^{\eps}}{1+a_1}]\big)(x_*)\le f(\lambda[\frac{\chi-\rho(\eps)\omega_0}{1-a_2}+dd^c\frac{(\varphi_2)_{\eps}}{1-a_2}])(x_*),\\
&f\big(\lambda[\frac{\chi+\rho(\eps)\omega_0}{1+a_1}+dd^c\frac{(\varphi_1)^{\eps}}{1+a_1}]\big)(x_*)\ge e^{G(x_*,\frac{(\varphi_1)^{\eps}(x_*)}{1+a_1})}-\rho_1(\eps,a_1),\\
&f\big(\lambda[\frac{\chi-\rho(\eps)\omega_0}{1-a_2}+dd^c\frac{(\varphi_2)_{\eps}}{1-a_2}])\big)(x_*)\le e^{G(x_*,\frac{(\varphi_2)_{\eps}(x_*)}{1-a_2})}+\rho_2(\eps,a_2).
\end{split}
\end{equation}
Combining the three inequalities,  we see that:
\begin{equation}\label{4.21}
e^{G(x_*,\frac{(\varphi_1)^{\eps}(x_*)}{1+a_1})}-\rho_1(\eps,a_1)\le e^{G(x_*,\frac{(\varphi_2)_{\eps}(x_*)}{1-a_2})}+\rho_2(\eps,a_2).
\end{equation}
On the other hand,  $\frac{(\varphi_1)^{\eps}(x_*)}{1+a_1}-\frac{(\varphi_2)_{\eps}(x_*)}{1-a_2}=\kappa_{\eps,a_1,a_2}\rightarrow \kappa_0>0$ as $\eps,\,a_1,\,a_2\rightarrow 0$.  This is clearly inconsistent with (\ref{4.21}) when $\eps,\,a_1,\,a_2$ are all small enough.
\end{proof}
In general,  when we don't necessarily know that $(\varphi_1)^{\eps}$ and $(\varphi_2)_{\eps}$ are punctually second order differentiable at $x_*$,  we need a perturbation argument which was first due to Crandall,  Ishii and Lions,  and we have:
\begin{lem}\label{l5.4N}
We choose normal coordinate near $x_*$ such that $x_*$ is given by $z=0$.  Then there exists a neighborhood
$U_0$ of $x_*$,  and 
there exists a sequence $p_k\in \bC^n$,  $p_k\rightarrow 0$,  and a sequence $\delta_k>0,\,\delta_k\rightarrow 0$,  such that one can find a sequence of points $x_k\in U_0$ with the following properties hold:
\begin{enumerate}
\item $x_k\rightarrow x_*$ as $k\rightarrow \infty$,
\item $\frac{(\varphi_1)^{\eps}}{1+a_1}-\frac{(\varphi_2)_{\eps}}{1-a_2}-<p_k,z>-\delta_k|z|^2$ has local maximum at $x_k$,
\item Both $(\varphi_1)^{\eps}$ and $(\varphi_2)_{\eps}$ are punctually second order differentiable at $x_k$.
\end{enumerate} 
\end{lem}
In the proof of Theorem \ref{t4.1} for the general case,  all we need to do is to consider a small perturbation of $(\varphi_1)^{\eps}$,  and one has:
\begin{lem}\label{l4.14N}
Define $\psi_{\eps,a_1,k}(z)=\frac{(\varphi_1)^{\eps}}{1+a_1}-<p_k,z>-\delta_k|z|^2$ on $U_0$.  Then for large enough $k$ ($\eps,\,a_1$ is fixed now) $\psi_{\eps,a_1,k}$ solves the following inequalities on $U_0$ in the viscosity sense:
\begin{equation}\label{l5.7N}
\begin{split}
&\lambda[\frac{\chi+2\rho(\eps)\omega_0}{1+a_1}+dd^c\psi_{\eps,a_1,k}]\in \Gamma.\\
&
f\big(\lambda[\frac{\chi+2\rho(\eps)\omega_0}{1+a_1}+dd^c\psi_{\eps,a_1,k}]\big)\ge e^{G(x,\psi_{\eps,a_1,k})}-2\rho_1(a_1,\eps).
\end{split}
\end{equation}
\end{lem}
The proof of Lemma \ref{l5.4N} and Lemma \ref{l4.14N} can be found in Lemma 5.16 and 5.17 in \cite{CX}.  The argument that Lemma \ref{l5.4N} and \ref{l4.14N} can be combined to prove Theorem \ref{t4.1} in the general case can be found in \cite{CX} (after Lemma 5.17).

\subsection{when the right-hand side does not depend on $\varphi$}
In this subsection,  we will assume that the right hand side $G$ depends only on $x$.  In this case,  one can only hope to solve $F(\chi+dd^c\varphi)=e^{G+c}$,  for some constant $c$.
Even though we can show that there is a unique $c\in \bR$ that makes this equation solvable in the viscosity sense,  we still don't have a good enough understanding of this constant,  which is the main hurdle to a proof of uniqueness of viscosity solutions in this case.
Let us start with observing the monotonicity of the constant,  in terms of the right hand side.  More specifically,  we have:
\begin{prop}\label{p4.15}
Let $G_1,\,G_2\in C(M)$ with $G_1\ge G_2$.  Assume that there exist $c_i\in \bR,\,\varphi_i\in C(M),\,i=1,\,2$,  that solve $F(\chi+dd^c\varphi_i)=e^{G_i+c_i}$ in the viscosity sense.  Then one has $c_1\le c_2$.
\end{prop}

Heuristically if this is false,  namely $c_1>c_2$.  This would imply that $e^{G_1+c_1}>e^{G_2+c_2}+\delta_0$ on $M$,  for some $\delta_0>0$.  This would lead to a contradiction if one evaluates at $x_0$,  where $\varphi_1-\varphi_2$ achieves maximum.  At this point,  one would have $\chi+dd^c\varphi_1\le \chi+dd^c\varphi_2$,  so that $F(\chi+dd^c\varphi_1)\le F(\chi+dd^c\varphi_2)$ at $x_0$.  In order to prove this rigorously, we can regularize $\varphi_1$ and $\varphi_2$ similar to the proof of the Theorem \ref{t4.1}.  
The detailed argument is identical to Proposition 5.6 of \cite{CX}.

A direct consequence of the above proposition is that there is a unique constant $c$ that allows for a viscosity solutions:
\begin{cor}
Let $G\in C(M)$,  there is at most one constant $c\in \bR$,  such that $F(\chi+dd^c\varphi)=e^{G+c}$ is solvable in the viscosity sense.
\end{cor}
Because of this result,  we may simply denote this constant to be $c(G)$.  Proposition \ref{p4.15} then implies that $c(G_1)\le c(G_2)$ whenever $G_1\ge G_2$.

The question that is of crucial importance to us is the following:
\begin{q}\label{q4.17}
Assume that $G_1\ge G_2$,  and $G_1\neq G_2$.  We define $c(G)$ using (\ref{1.6NNN}).  Do we actually have $c(G_1)<c(G_2)$?
\end{q}
In order to justify its importance,  we are going to show the uniqueness of viscosity solutions,  assuming we have an affirmative answer to Question \ref{q4.17}.  More precisely:
\begin{thm}\label{t4.2}
Let $G\in C(M)$.  Assume that there exists a strict subsolution in the viscosity sense as defined by Definition \ref{d2.2}.  Assume also that for any $G'\in C(M),\,G'\le G$ and $G'\neq G$,  one has $c(G')>c(G)$ where $c(G')$ and $c(G)$ are defined by (\ref{1.6NNN}).  Then there is at most one viscosity solution to $F(\chi+dd^c\varphi)=e^{G+c}$.
\end{thm}
\begin{proof}
Let $\varphi_1$ and $\varphi_2$ be two viscosity solutions to the equation.  Denote $E=\{x\in M:\varphi_2(x)-\varphi_1(x)=\min_M(\varphi_2-\varphi_1)\}$.  Clearly $E$ is a compact subset of $M$ and we will be done if we can show $E=M$.  Assume otherwise,  we can take $\delta>0$ small enough,  such that $E_{\delta}\neq M$,  where $E_{\delta}$ denotes the $\delta$-neighborhood of $E$.

Now we can define $\tilde{G}\in C(M)$ as follows: $e^{\tilde{G}}=e^G$ on $E_{\frac{\delta}{2}}$,  $e^{\tilde{G}}=e^G-\eps_0$ for some small positive constant $\eps$ outside $E_{\delta}$,  and $e^G -\eps_0 \le e^{\tilde{G}}\le e^G$ on $M$.  Let $\eta$ be a viscosity solution to:
\begin{equation}\label{5.8N}
F(\chi+dd^c\eta)=e^{\tilde{G}+\tilde{c}},\,\,\lambda[\chi+dd^c\eta]\in \Gamma,\,\,\sup_M\eta=0.
\end{equation}
Here $\tilde{c}$ is defined by (\ref{1.6NNN}).  We will have to justify this by showing that 

From Lemma \ref{l4.2NNew},  we see that with $\eps_0>0$ choose small enough,  $c$ will be close to $\tilde{c}$,  and we already knew $G$ is close to $\tilde{G}$ up to an error of $\eps_0$.  Since we already have a viscosity solution $\underline{u}$ that satisfies the following in the viscosity sense: for some $\delta_0>0$,
\begin{equation}
f_{\infty}\big(\lambda[\chi+dd^c\underline{u}]\big)\ge e^{G+c}+\delta_0,\,\,\,\lambda[\chi+dd^c\underline{u}]\in \Gamma_{\infty}.
\end{equation} 
With $\eps_0>0$ small enough,  this easily translates to:
\begin{equation}
f_{\infty}\big(\lambda[\chi+dd^c\underline{u}]\big)\ge e^{\tilde{G}+\tilde{c}}+\frac{1}{2}\delta_0,\,\,\,\lambda[\chi+dd^c\underline{u}]\in \Gamma_{\infty}.
\end{equation}
Using this,  we see that the equation (\ref{5.8N}) can be solved in the viscosity sense.
 From the assumption,  we know that $\tilde{c}>c$.  Let $0<r<1$,  we consider the minimum of $\varphi_2-((1-r)\varphi_1+r\eta)$ and assume that it is achieved at $x_{\delta,r}$. Using a perturbation argument as before, we can assume that $\varphi_1$, $\varphi_2$ and $\eta$ are puntually second order differentiable at $x_{\delta,r}$ without loss of generality (the full rigorous argument follows the proof of Theorem 5.2 in \cite{CX}).  Evaluating at $x_{\delta,r}$,  we would have:
\begin{equation*}
f\big(\lambda[\chi+dd^c\varphi_2]\big)\ge f\big(\lambda[\chi+dd^c(1-r)\varphi_1+r\eta]\big)\ge (1-r)f(\lambda[\chi+dd^c\varphi_1])+rf(\lambda[\chi+dd^c\eta]).
\end{equation*}
This would mean that:
\begin{equation}\label{4.24}
e^{G+c}(x_{\delta,r})\ge (1-r)e^{G+c}(x_{\delta,r})+re^{\tilde{G}+\tilde{c}}(x_{\delta,r}).
\end{equation}
On the other hand,  for fixed $\delta$,  the minimum of $\varphi_2-((1-r)\varphi_1+r\eta)$ can only be achieved in $E_{\frac{\delta}{2}}$,  as long as $r$ is small enough,  but then one would have $e^{\tilde{G}}(x_{\delta,r})=e^{G}(x_{\delta,r})$,  and this is inconsistent with (\ref{4.24}). 
\end{proof}

\section{Acknowlegement}
This work is partially supported by Simons Foundation Award with ID:605796.


\begin{thebibliography}{9}

\bibitem{C}
G. Chen:
\newblock
The J-equation and the supercritical deformed Hermitian–Yang–Mills equation.
\newblock
Inventiones mathematicae 225 (2021): 529-602.

\bibitem{CX}
J.Cheng, and Y.Xu:
\newblock
Viscosity solution to complex Hessian equations on compact Hermitian manifolds.
\newblock
 arXiv preprint arXiv:2406.00953 (2024).

\bibitem{CX2}
J.Cheng, and Y.Xu:
\newblock
Regularization Of $ m $-subharmonic Functions And H\" Older Continuity.
\newblock
arXiv preprint arXiv:2208.14539 (2022).

\bibitem{CM}
J.C. Chu and N. McCleerey:
\newblock
Fully non-linear degenerate elliptic equations in complex geometry.
\newblock
 Journal of Functional Analysis 281.9 (2021): 109176.
 
 \bibitem{DDT}
S.Dinew, H. S. Do, and T. D. To:
\newblock
A viscosity approach to the Dirichlet problem for degenerate complex Hessian-type equations.
\newblock
 Analysis $\&$ Pde 12.2 (2018): 505-535.


\bibitem{EGZ}
P.Eyssidieux, V. Guedj, and A. Zeriahi:
\newblock
Viscosity solutions to degenerate complex monge‐Ampère equations.
\newblock
Communications on pure and applied mathematics 64.8 (2011): 1059-1094.

\bibitem{Guan}
B.  Guan:
\newblock The Dirichlet problem for a class of fully nonlinear elliptic equations:
\newblock Comm.  PDE,  vol 19,  issue 3-4(1994),  399-416.

\bibitem{GPT}
B.  Guo,  D. H.  Phong and F.Tong:
\newblock On $L^{\infty}$ estimates for complex Monge-Ampere equations.
\newblock Ann of Math,  second series,  vol 198,  no.  1(2023),  393-418.

\bibitem{GS}
B. Guo, and J. Song:
\newblock
Sup-slopes and sub-solutions for fully nonlinear elliptic equations.
\newblock
arXiv preprint arXiv:2405.03074 (2024).

\bibitem{HLP}
F.R. Harvey, H. B. Lawson Jr, and S. Pliś:
\newblock
The Richberg technique for subsolutions.
\newblock
Communications in Analysis and Geometry 28.8 (2020): 1787-1806.

\bibitem{K}
S. Kołodziej, 
\newblock
The complex Monge-Amp\'ere equation.
\newblock
 Acta Math,  vol 180,  no.1(1998): 69-117.

\bibitem{KN0}
S.Kołodziej, and N. C. Nguyen:
\newblock 
Weak solutions to the complex Monge-Ampère equation on Hermitian manifolds.
\newblock arXiv preprint: 1312.5419.


\bibitem{L}
H. C. Lu:
\newblock
Viscosity solutions to complex Hessian equations.
\newblock
Journal of Functional Analysis 264.6 (2013): 1355-1379.


\bibitem{S}
G.  Székelyhidi:
\newblock
Fully non-linear elliptic equations on compact Hermitian manifolds.
\newblock
Journal of Differential Geometry 109.2 (2018): 337-378.

\bibitem{SS}
Z.N.Sui and W. Sun: 
\newblock
On $L^{\infty}$ estimate for complex Hessian quotient equations on compact K\"ahler manifolds.
\newblock
 The Journal of Geometric Analysis 33.6 (2023): 165.


\bibitem{S2}
W. Sun:
\newblock 
The boundary case for complex Monge-Amp\'ere type equations.
\newblock
 arXiv preprint arXiv:2305.02576 (2023).
 

\bibitem{T}
N.S.Trudinger:
\newblock
On the Dirichlet problem for Hessian equations.
\newblock
Acta.  Math,  vol 175,  no. 2(1995), 151-164.

\bibitem{Y}
R.Yuan:
\newblock
Fully non-linear elliptic equations on complex manifolds.
\newblock
 arXiv preprint arXiv:2404.12350 (2024).

\end{thebibliography}
\end{document}